\documentclass[12pt
,final
]{article}
\usepackage[latin1]{inputenc} 

\usepackage{enumerate}%
\usepackage{floatflt}%
\usepackage{amssymb}%
\usepackage{amsmath}%
\usepackage{amssymb}%
\usepackage{amsthm}%
\usepackage{amsxtra}%
\usepackage{amsmath}%
\usepackage{amscd}%
\usepackage{graphicx}\graphicspath{{./pics/}}%
\usepackage{psfrag}%
\usepackage{adjustbox}%
\usepackage{tikz}%
\usetikzlibrary{%
  graphs,%
  graphs.standard,%
  arrows,arrows.meta,%
  automata,%
  decorations.pathmorphing,%
  backgrounds,%
  positioning,%
  fit,%
  petri,%
  shapes.misc,%
  shapes.multipart,%
  shapes,%
  calc%
}%
\usepackage{latexsym}%
\usepackage{color}%
\usepackage{a4wide}%
\usepackage[labelformat=simple]{subcaption}%
\usepackage{hyperref}%
\usepackage[ruled]{algorithm2e}%
\usepackage{fancyhdr}%
\fancypagestyle{plain}{%
}

\date{}
\newtheorem{theorem}{Theorem}
\newtheorem{lemma}[theorem]{Lemma}

\newtheorem{corollary}[theorem]{Corollary}
\newtheorem{fact}[theorem]{Fact}
\newtheorem{claim}[theorem]{Claim}

\newtheorem{definition}[theorem]{Definition}

\def\squareforqed{\hbox{\rlap{$\sqcap$}$\sqcup$}}
\def\myqed{\ifmmode\squareforqed\else{\unskip\nobreak\hfil
\penalty50\hskip1em\null\nobreak\hfil\squareforqed
\parfillskip=0pt\finalhyphendemerits=0\endgraf}\fi}

\newcommand{\T}{{\mathcal D}}
\newcommand{\deficiency}{\mathrm{def}}
\newcommand{\dist}{\mathrm{dist}}

\title{Lower bounds on collective additive spanners} 

\author{%
  Derek G. Corneil\thanks{Department of Computer Science, %
    University of Toronto, Toronto, Ontario, Canada, %
    \href{mailto:dgc@cs.toronto.edu}{dgc@cs.toronto.edu}} %
\and %
  Feodor F. Dragan\thanks{Department of Computer Science, %
    Kent State University, Kent, Ohio, U.S.A., %
    \href{mailto:dragan@cs.kent.edu}{\{dragan,yxiang\}@cs.kent.edu}}%
  \and %
  Ekkehard K{\"o}hler\thanks{Institut f{\"u}r Mathematik, %
    Brandenburgische Technische Universit{\"a}t, Cottbus-Senftenberg, Germany, %
    \href{mailto:ekkehard.koehler@b-tu.de}{ekkehard.koehler@b-tu.de}}%
  \and %
  Yang Xiang$^{~\dag}$%
}%

\begin{document}

\maketitle

\begin{abstract}
  In this paper we present various lower bound results on collective
  tree spanners and on spanners of bounded treewidth. A graph $G$ is
  said to admit a system of $\mu$ collective additive tree
  $c$-spanners if there is a system $\cal{T}$$(G)$ of at most $\mu$
  spanning trees of $G$ such that for any two vertices $u,v$ of $G$ a
  tree $T\in \cal{T}$$(G)$ exists such that the distance in $T$
  between $u$ and $v$ is at most $c$ plus their distance in $G$.  A
  graph $G$ is said to admit an additive $k$-treewidth $c$-spanner if
  there is a spanning subgraph $H$ of $G$ with treewidth $k$ such that
  for any pair of vertices $u$ and $v$ their distance in $H$ is at
  most $c$ plus their distance in $G$.  Among other results, we show
  that:
  \begin{itemize}
  \item[-] Any system of collective additive tree $1$--spanners must
    have $\Omega(\sqrt[3]{\log n})$ spanning trees for some unit
    interval graphs;

  \item[-] No system of a constant number of collective additive tree
    $2$-spanners can exist for strongly chordal graphs;

  \item[-] No system of a constant number of collective additive tree
    $3$-spanners can exist for chordal graphs;

  \item[-] No system of a constant number of collective additive tree
    $c$-spanners can exist for weakly chordal graphs as well as for
    outerplanar graphs for any constant $c\geq 0$;

  \item[-] For any constants $k \ge 2$ and $c \ge 1$ there are graphs
    of treewidth $k$ such that no spanning subgraph of treewidth $k-1$
    can be an additive $c$-spanner of such a graph.
  \end{itemize}
All these lower bound results apply also to general graphs. Furthermore, they 
complement known upper bound results with tight 
  lower bound results.\footnote{Some of
    the results on unit interval graphs in this paper were announced in~\cite{CorneilDKY05}.}
\end{abstract}
    

\newcommand{\w}{{\sf w}}

\section{Introduction}\label{intro}

A spanning subgraph $H$ of a graph $G$ is called a
{\em spanner} of $G$ if $H$ provides a ``good'' approximation of the
distances in $G$.  More formally, for $c\geq 0$, $H$ is called an {\it
  additive $c$-spanner} of $G$ if for any pair of vertices $u$ and $v$
their distance in $H$ is at most $c$ plus their distance in $G$
\cite{LieShe}. In this case, $c$ is called a {\em surplus}. If $H$ is
a spanning tree of $G$ then $H$ is called an {\it additive tree
  $c$-spanner} of $G$ \cite{KLM}.  The {\it sparsest additive
  $c$-spanner problem} asks for a given graph $G$ to find an additive
$c$-spanner $H$ with the smallest number of edges.  (Similar
definitions can be given for multiplicative $c$-spanners
\cite{CaiC95,Chew,PeSc,PelUll}; however since we are only concerned
with additive spanners, we will often omit ``additive''.)

Motivated by problems of designing compact and efficient routing and
distance labeling schemes in networks, in \cite{DYL}, a new notion of
{\em collective tree spanners}
of graphs was introduced.  This notion is slightly {\sl weaker} than
the one of a tree spanner and slightly {\sl stronger} than the notion
of a sparse spanner.  We say that a graph $G=(V,E)$ {\em admits a
  system of $\mu$ collective additive tree $c$-spanners} if there is a
system $\cal{T}$$(G)$ of at most $\mu$ spanning trees of $G$ such that
for any two vertices $u,v$ of $G$ a spanning tree $T\in \cal{T}$$(G)$
exists such that the distance in $T$ between $u$ and $v$ is at most
$c$ plus their distance in $G$. We say that system ${\cal T}(G)$
collectively $+c$-spans (or simply $c$-spans as we are only concerned
with additive spanners) the graph $G$. Clearly, if $G$ admits a system
of $\mu$ collective additive tree $c$-spanners, then $G$ admits an
additive $c$-spanner with at most $\mu\times (n-1)$ edges (take the
union of all those trees), and if $\mu=1$ then $G$ admits an additive
tree $c$-spanner. Note also that any graph on $n$ vertices admits a
system of at most $n-1$ collective additive tree 0-spanners (take
$n-1$ Breadth-First-Search--trees (also known as \emph{shortest path
  trees}) rooted at different vertices of $G$).

One of the main motivations to introduce this new concept stems from
the problem of designing compact and efficient routing
and distance labeling schemes in graphs.  In
\cite{FrGa01,ThZw01}, a shortest path routing labeling scheme for
trees is described that assigns each vertex of an $n$-vertex tree an
$O(\log n)$-bit label. Given the label of a source vertex and the
label of a destination, it is possible to compute in constant time,
based solely on these two labels, the neighbor of the source that
heads in the direction of the destination. If an $n$-vertex graph $G$
admits a system of $\mu$ collective additive tree $c$-spanners, then
$G$ admits a routing labeling scheme of deviation (i.e., additive
stretch) $c$ with addresses and routing tables of size $O(\mu
\log^2n)$ bits per vertex. Once computed by the sender in $\mu$ time
(by choosing for a given destination an appropriate tree from the
collection to perform routing), headers of messages never change, and
the routing decision is made in constant time per vertex (for details
see \cite{DYC,DYL}).
Recently, collective tree spanners found applications also in the
problem of location-aware synchronous rendezvous in networks
\cite{CollinsCGKM11} and in the k-Server problem under the advice
model of computation when the underlying metric space is sparse
\cite{GuptaKL13}.

To date, a number of papers establishing various upper bound results
and some ``easy-to-derive'' lower bound results on collective tree
spanners of particular classes of graphs have been published
\cite{CorneilDKY05,DraganA14,DraganCKX12,DraganY10,DYC,DYL,DraganYX08,Gup,YanXD12}.
The definitions of the various families of graphs studied in this
paper appear in Subsection \ref{def}; for the definitions of other
mentioned families either see the referenced manuscript or
\cite{Classes}.  A system of at most $O(\log n)$ collective additive
tree $2$--spanners exists for all chordal graphs \cite{DYL}, chordal
bipartite graphs \cite{DYL}, homogeneously orderable graphs
\cite{DraganYX08}, and circle graphs \cite{DraganCKX12}. A system of
two collective additive tree $2$--spanners exists for all AT-free
graphs which contains interval graphs, permutation graphs and
cocomparability graphs \cite{DYC} and for all circular-arc graphs
\cite{DYL}.  These results were complemented by lower bounds, which
say that any system of collective additive tree $1$--spanners must
have $\Omega(\sqrt{n})$ spanning trees for some chordal graphs and
$\Omega(n)$ spanning trees for some chordal bipartite graphs and some
cocomparability graphs \cite{DYL}.  Furthermore, any $k$-chordal graph
admits a system of at most $\log_{2} n$ collective additive tree
$(2\lfloor k/2\rfloor)$--spanners \cite{DYL}, any graph with
tree-breadth $\rho$ has a system of at most $\log_2n$ collective
additive tree $(2\rho\log_2n)$-spanners \cite{DraganA14}, and any
graph having a dominating shortest path admits a system of two
collective additive tree 3-spanners and a system of five collective
additive tree 2-spanners \cite{DYC}. All collective tree spanners
mentioned above can be constructed in linear or low-polynomial time.
It is interesting to note also that there is a polynomial time
algorithm that, given an $n$-vertex graph $G$ admitting a
(multiplicative) tree $c$-spanner, constructs a system of at most
$\log_2 n$ collective additive tree $O(c\log n)$-spanners of $G$ (see
\cite{DraganA14} for details). 
This result was improved in a recent paper \cite{BeRa22}: for every graph $G$ admitting a (multiplicative) tree $c$-spanner, one can efficiently construct one spanning tree $T$ such that $d_T(x,y)\le d_G(x,y)+O(c\log n)$ holds for every $x,y\in V$ (turning a multiplicative distortion into an additive surplus with an additional logarithmic factor). 
Recall that the problem of finding for
a given graph a tree $c$-spanner with minimum $c$ is NP-hard
\cite{CaiC95}.



Collective tree spanners of {\em Unit Disk Graphs} (UDGs) (which often
model wireless ad hoc networks) were investigated in \cite{YanXD12}.
It was shown that every $n$-vertex UDG $G$ admits a system ${\cal
  T}(G)$ of at most $2 \log_{\frac{3}{2}} n+2$ spanning trees of $G$
such that, for any two vertices $x$ and $y$ of $G$, there exists a
tree $T$ in ${\cal T}(G)$ with $\dist_T(x,y)\leq 3\cdot \dist_G(x,y)
+12$.  That is, the distances in any UDG can be approximately
represented by the distances in at most $2 \log_{\frac{3}{2}} n+2$ of
its spanning trees.  Based on this result a new {\em compact and low
  delay routing labeling scheme} was proposed for Unit Disk Graphs
(see \cite{YanXD12} for more details). 

There are several related lines of research. Collective {\em multiplicative} tree spanners were considered in \cite{chang2025lighttreecoversrouting,gupta2006oblivious,YanXD12}. They called in \cite{gupta2006oblivious} {\em spanning tree covers}. Approximation of metric spaces by a small number of tree metrics, so-called {\em tree covers}, were considered in \cite{BARTAL202226,10353154,chang_et_al:LIPIcs.SoCG.2024.37} (see also papers cited therein). Probabilistic approximation of metric spaces by a small number of tree metrics were considered in \cite{bartal1996probabilistic,charikar1998approximating}.  
 
\subsection{Results of this paper}\label{results}

In this paper we establish various non-trivial lower bounds on
collective tree spanners of particular classes of graphs such as
(unit) interval graphs, (strongly) chordal graphs, weakly chordal
graphs, outerplanar graphs, and AT-free graphs.  The definitions of
all these graph classes appear in Subsection~\ref{def}. The following
inclusions hold amongst them.

\medskip

{\centering \small{\sf
    \begin{tabular}{ccccccccc}
      unit interval&$\subset$&interval&$\subset$&strongly chordal&$\subset$&chordal&$\subset$&weakly chordal\\
      &         &        &$\subset$&cocomparability &$\subset$&AT-free&         &              \\
      &         &        &         &                &         &       &         &              \\
                            
      outerplanar&$\subset$&planar& & & & & & \\               
    \end{tabular}}}

\bigskip

Note that in many cases our lower bounds correspond to known upper
bounds; these results are summarized in Table \ref{tbl:summary}. Clearly, all these lower bound results apply also to general graphs. 

No system of a constant number of collective additive tree
$1$-spanners can exist for unit interval graphs. In fact, any system
of collective additive tree $1$--spanners must have
$\Omega(\sqrt[3]{\log n})$ spanning trees for some unit interval
graphs.  The same lower bound holds for interval graphs, strongly
chordal graphs and Unit Disk Graphs as they all contain unit interval
graphs. These results are presented in Section~\ref{lxwer}.
  

In Section~\ref{master} we state and prove a meta theorem that shows
that if a particular family $\cal{F}$ of graphs contains a ``gadget''
with specific properties (expressed as functions of $c$ and $d$) then
no $d$ trees can collectively $+c$ span a sufficiently large $\cal{F}$
graph.  The results of Sections~\ref{sec:coll-tree-spann}
and~\ref{sec:lower-bound-coll} contain corollaries of this meta
theorem applied to (weakly) chordal and strongly chordal graphs
respectively; see Table \ref{tbl:summary}.

\begin{table}
  \small
  \centering
  \begin{tabular}{l c c c}
    \hline\hline
    Graph Class & Surplus & $\#$ of Trees (l.b. $\leftrightarrow$ u.b.) & Source  \\
    \hline\hline
    Interval & +2 & 1 & \cite{KLM}\\
             & +1 & $\Omega(\sqrt[3]{\log n})  \leftrightarrow  O(\log D)$ &[here],\cite{CorneilDKY05}  \\
    \hline
    AT-free          & +3 & 1 & \cite{KLM} \\
                     & +2 & 2 & \cite{DYC} \\
                     & +1 & $\Theta({n})$ & \cite{DYL}\\                     
    \hline
    Strongly Chordal & +3 & 1 & \cite{BrandstadtCD99} \\
                     & +2 & no constant $\leftrightarrow  \log_2 n$ & [here],\cite{DYL} \\
                     & +1 & $\Omega(\sqrt[3]{\log n}) \leftrightarrow  O(n)$ & [here],[trivial] \\
    
    \hline
    Chordal          & $+c, c>3$ &  constant? & [open problem] \\
                     & $+c, c=2,3$ & no constant $\leftrightarrow  \log_2 n$ & [here],\cite{DYL} \\
                     & +1 & $\Omega(\sqrt{n})  \leftrightarrow  O(n)$ & \cite{DYL},[trivial] \\
    \hline
    Weakly Chordal   & $+c, \forall c\geq 2$ & no constant $\leftrightarrow  4\log_2 n$ & [here],\cite{DraganY10} \\
                     & +1 & $\Theta({n})$ & \cite{DYL}\\  
    \hline
    Outerplanar      & $+c, \forall c\geq 0$ & no constant $\leftrightarrow  3\log_2 n$ & [here],\cite{DraganY10,Gup}  \\                                                              \hline\hline
  \end{tabular}  
  \caption{Summary of upper and lower bound results on collective additive tree spanners.}\label{tbl:summary}
\end{table} 


Furthermore, in Section~\ref{sec:spanners-with-high}, we establish a
lower bound for a new type of spanners. In \cite{DrFoGo}, {\em
  spanners of bounded treewidth} were introduced, motivated by the
fact that many algorithmic problems are tractable on graphs of bounded
treewidth, and a spanner $H$ of $G$ with small treewidth can be used
to obtain an approximate solution to a problem on $G$. In particular,
efficient and compact distance and routing labeling schemes are
available for bounded treewidth graphs (see, e.g.,
\cite{DraganY10,Gup} and the papers cited therein), and they can be
used to compute approximate distances and route along paths that are
close to shortest in $G$.  The {\em additive $k$-treewidth
  $c$-spanner} problem asks, for a given graph $G$, an integer $k$ and
a number $c\geq 0$, whether $G$ admits an additive $c$-spanner of
treewidth at most $k$. Similarly, the {\em multiplicative
  $k$-treewidth $c$-spanner} problem can be defined.  Every connected
graph with $n$ vertices and at most $n-1+m$ edges is of treewidth at
most $m+1$ and hence this problem is a generalization of the tree
$c$-spanner and the sparsest $c$-spanner problems. Furthermore,
$t$-spanners of bounded treewidth have much more structure to exploit
algorithmically than sparse $t$-spanners (which have a small number of
edges but may lack other nice structural properties).

The multiplicative $k$-treewidth $c$-spanner problem was considered in
\cite{DrFoGo} and \cite{FominGL11}. It was shown that the problem is
linear time solvable for every fixed constants $c$ and $k$ on the
class of apex-minor-free graphs \cite{DrFoGo}, which includes all
planar graphs and all graphs of bounded genus, and on the graphs with
bounded degree \cite{FominGL11}. See also \cite{CohenAddad2020OnLS,10353171,FL2022} for low treewidth embeddings of planar and minor-free metrics. 

In Section~\ref{sec:spanners-with-high}, we show the following lower
bound result: for any constants $k \ge 2$ and $c \ge 1$ there are
graphs of treewidth $k$ such that no spanning subgraph of treewidth
$k-1$ can be a $c$-spanner of such a graph.

\subsection{Notation and definitions}\label{def}

All graphs occurring in this paper are connected, finite, undirected,
loopless and without multiple edges. In a graph $G=(V,E)\: (n=|V|,
m=|E|)$ the {\it length} of a path from a vertex $v$ to a vertex $u$
is the number of edges in the path. The {\it distance}
$\dist_G(u,v)$ between the vertices
$u$ and $v$ is the length of a shortest path connecting $u$ and $v$.


An independent set of three vertices such that each pair is joined by
a path that avoids the neighborhood of the third is called an {\em
  asteroidal triple} (AT). A graph $G$ is an {\em AT-free graph} if it
does not contain any asteroidal triples \cite{COS96}. A graph is {\em
  chordal} if it does not contain any induced cycles of length greater
than 3.  A graph is an \emph{interval graph} if one can associate with
each vertex an interval on the real line such that two vertices are
adjacent if and only if the corresponding intervals have a nonempty
intersection. Furthermore, an interval graph is a \emph{unit interval
  graph} if all intervals are of the same length.  Unit interval
graphs are equivalent to \emph{proper interval graphs} where no
interval can properly contain any other interval.  It is well known
that a graph is an interval graph if and only if it is both a chordal
graph and an AT-free graph \cite{LB}.

A graph is an {\em outerplanar graph} if it can be drawn in the plane
without crossings in such a way that all of the vertices belong to the
unbounded face of the drawing.  A graph $G$ is a {\em cocomparability
  graph} if it admits a vertex ordering $\sigma=[v_1,v_2,\dots, v_n]$,
called a {\em cocomparability ordering}, such that for any $i<j<k$, if
$v_i$ is adjacent to $v_k$ then $v_j$ must be adjacent to at least one
of $v_i$, $v_k$. It is known that the class of cocomparability graphs
is a superclass of interval graphs and a subclass of AT-free graphs
(see \cite{Gol-book}).

A graph $G$ is {\em strongly chordal} if it is
chordal and every cycle of even length ($\ge 6$)
in $G$ has an odd chord, i.e., an edge that connects two vertices that
are an odd distance ($> 1$) apart from each
other in the cycle.
All strongly chordal graphs are chordal and they contain all interval
graphs \cite{faber-strongly-chordal-83}.

A graph is \emph{weakly chordal} (also called \emph{weakly
  triangulated}) if neither $G$ nor its complement $\overline{G}$
contain an induced hole (cycle of size at least 5).  A graph $G$ is
\emph{house-hole-domino-free (HHD-free)} if it does not contain the
house, the domino, and holes as induced subgraphs (see
Fig.~\ref{hhd-fig}).  Clearly, chordal graphs are strictly contained
in both weakly chordal and HHD-free graphs.

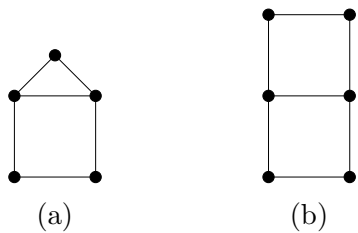
\begin{figure}[htb]
  \begin{center} %
    \subcaptionbox{\label{fig:sub:house}}{      \begin{tikzpicture}[new set=import nodes,scale=0.18]%
        \begin{scope}[nodes={set=import nodes},%
          nodes={%
            inner sep=1.5pt,%
            draw,%
            circle,%
            draw=black,%
            fill=black%
          }%
          ]%
          \node (1) at (0,0) {$$}; %
          \node (2) at (6,0) {$$}; %
          \node (3) at (0,6) {$$}; %
          \node (4) at (6,6) {$$}; %
          \node (5) at (3,9) {$$}; %
        \end{scope}

        \graph[edges={%
          black,shorten >=-0.5pt,shorten <=-0.5pt,very thin %
        }] {%
          (import nodes); %

          1 -- 2;%
          1 -- 3;%

          2 --  4;%

          3 -- 4;%
          3 -- 5;%

          4 -- 5;%
        };%
      \end{tikzpicture}}%
    \hspace{2cm}%
    \subcaptionbox{\label{fig:sub:domino}}{      \begin{tikzpicture}[new set=import nodes,scale=0.18]%
        \begin{scope}[nodes={set=import nodes},%
          nodes={%
            inner sep=1.5pt,%
            draw,%
            circle,%
            draw=black,%
            fill=black%
          }%
          ]%
          \node (1) at (0,0) {$$}; %
          \node (2) at (6,0) {$$}; %
          \node (3) at (0,6) {$$}; %
          \node (4) at (6,6) {$$}; %
          \node (5) at (0,12) {$$}; %
          \node (6) at (6,12) {$$}; %
        \end{scope}

        \graph[edges={%
          black,shorten >=-0.5pt,shorten <=-0.5pt,very thin %
        }] {%
          (import nodes); %

          1 -- 2;%
          1 -- 3;%

          2 --  4;%

          3 -- 4;%
          3 -- 5;%

          4 -- 6;%

          5 -- 6;%
        };%
      \end{tikzpicture}}%
    \end{center}%
    \caption{\label{hhd-fig} \subref{fig:sub:house} A house, %
    \subref{fig:sub:domino} A domino.}%
\end{figure}%


\section{Unit Interval Graphs }\label{lxwer}

Independently McKee~\cite{McK} and Kratsch et al.~\cite{KLM} showed
that no single tree can $c$-span a chordal graph for any constant $c$.
We now show a similar result for collectively 1-spanning a unit
interval graph.

\begin{theorem}\label{lower}
  No constant number of trees can collectively 1-span a unit interval
  graph.
\end{theorem}

\begin{proof}
  First we will show that two trees do not suffice and then show how
  to extend this result to any constant number of trees.

  The general ``gadget'' will be a $K_3$
  plus two independent universal vertices
  $x$ and $y$ (i.e. we have a $K_5$ with the edge $xy$ missing).  The
  vertices of the $K_3$ will be labelled 1, 2, 3.  Now make a
  sufficiently long chain of these gadgets by identifying the $y$
  vertex of a gadget with the $x$ vertex of its right neighbor.  It is
  straightforward to confirm that this graph $G$ is a unit interval
  graph.  Consider two trees $T_1$ and $T_2$ that supposedly
  collectively 1-span $G$.  By making the chain sufficiently long, by
  the ``pigeonhole principle'', we are guaranteed that there are three
  gadgets in $G$ namely, $A, B$ and $C$ where $A$ is left of $B$ which
  is left of $C$ such that:

  \begin{itemize}
  \item $T_1$ restricted to $A, B$ and $C$ is exactly the same
    spanning tree for all three gadgets.  Exactly the same means from
    the labelled vertex point of view,

  \item $T_2$ restricted to $A, B$ and $C$ is also exactly the same
    spanning tree for all three gadgets.  Note that $T_1$ restricted
    to $\{A,B,C\}$ is not necessarily the same as $T_2$ restricted to
    $\{A,B,C\}$.
  \end{itemize}

  The vertices in $A, B$ and $C$ will be denoted $A_x, B_3, C_y$,
  where, for example, $A_x$ refers to the $x$-vertex of $A$.  We say
  that a tree provides a 1-approximating path between two vertices if
  the distance between the vertices in the tree is at most 1 more than
  their distance in $G$.  We now show that in order for $T_1$ or $T_2$
  to provide such an approximating path, certain edges of $G$ must be
  present in the tree.

  \begin{claim}\label{lb}
    Let $i$ be an element of $\{1,2,3\}$. If either $T_1$ or $T_2$
    provides a 1-approxi\-mating path between $A_i$ and $C_i$, then it
    must contain the $xi$ and $yi$ edges in all of $A, B$ and $C$.
  \end{claim}

  \begin{proof}
    Without loss of generality, assume that $T_1$ provides the
    1-approximating path between $A_i$ and $C_i, i \in \{1,2,3\}$.
    Such a path requires either $A_i$ to be adjacent to $A_y$ and/or
    $C_i$ to be adjacent to $C_x$.  Without loss of generality, assume
    $C_i$ is adjacent to $C_x$; thus since $T_1$ when restricted to
    $A, B$ and $C$ is exactly the same, $A_i$ is adjacent to $A_x$ and
    $B_i$ is adjacent to $B_x$ as well.  We now show that in all three
    of $A, B$ and $C, i$ is also adjacent to $y$.  Suppose not; now in
    each gadget, the distance between $i$ and $y$ is at least 2 which
    means that the tree path between $A_i$ and $C_i$ must be at least
    2 greater than the distance in $G$ (since in $T_1$ the distance
    between $B_x$ and $B_y$ must be at least 3 by following the edge
    $B_xB_i$ and the path between $B_i$ and $B_y$).
  \end{proof}

  From the claim, it is clear that each of $T_1$ and $T_2$ can provide
  at most one path between $A_1, C_1$ or $A_2, C_2$ or $A_3, C_3$ and
  thus at least three trees are required to 1-approximate $G$.

  To generalize this argument, i.e., to show that at least $k$ trees
  are required, merely replace the $K_3$ in the gadget by a $K_k$.
  The same use of the claim shows that $k-1$ trees are not enough.
\end{proof}

Using the above construction one can also show the following theorem.

\begin{theorem}
  The number of trees needed to collectively tree 1-span an unit
  interval graph is $\Omega(\sqrt[3]{\log n})$.
\end{theorem}

\begin{proof}
  Consider the construction of the proof of Theorem~\ref{lower} for
  the case of $k$ trees.  That means each gadget in the constructed
  gadget graph $G$ consists of a $K_{k+1}$
  plus two nonadjacent universal vertices
  $x$ and $y$.  For the construction of the proof we have to make sure
  that there are three particular gadgets $A$, $B$, and $C$ such that
  each of the $k$ spanning trees of $G$, when restricted to these
  three gadgets are identical (note that the sub-tree covering $A$ in
  $T_1$ might be a different one than the sub-tree covering $A$ in
  $T_2$ but the sub-tree covering $A$ in $T_1$ is exactly the same as
  the sub-tree covering $B$ in $T_1$ and the sub-tree covering $C$ in
  $T_1$).

  How large must the graph $G$ be to guarantee this property?  For
  this we have to count how many different spanning trees exist for
  our gadget.  Since each gadget is a $K_{k+3}$ minus one edge, by
  Cayley's formula, there cannot be more than $(k+3)^{(k+1)}$
  different spanning trees for it.

  Hence, if the constructed gadget graph has at least $3
  (k+3)^{(k+1)}$ gadgets, then, by the pigeon hole principle, a
  spanning tree of this graph has to contain at least three gadgets
  $A$, $B$, $C$ that are covered by identical trees (in fact $2
  (k+3)^{(k+1)} + 1$ would suffice here).

  Next consider two spanning trees $T_1$, $T_2$ of $G$.  If we
  construct $G$ to consist of at least $3 (k+3)^{(k+1)} (k+3)^{(k+1)}$
  gadgets, then, by the pigeon hole principle, in the tree $T_1$ at
  least $3 (k+3)^{(k+1)}$ gadgets have to be covered by identical
  sub-trees.  Now, using the above argument, there have to be at least
  $3$ of those $3 (k+3)^{(k+1)}$ gadgets that are also covered by
  identical trees in the second tree $T_2$.

  If we now apply this argument $k$ times then it follows that a
  gadget graph consisting of $3 (k+3)^{(k+1)} (k+3)^{(k+1)k} = 3
  (k+3)^{(k+1)^2}$ gadgets has to contain at least $3$ gadgets
  that are covered in each of the $k$ trees by identical sub-trees.
  Hence these $k$ trees can not collectively $1$-span the gadget graph
  $G$.

  To get a lower bound on the number of trees needed to collectively
  1-span a unit-interval graph with $n$ vertices, let $n$ be the
  number of vertices of the gadget graph $G$.  Since each gadget
  consists of a $K_{k+1}$ plus the two independent universal vertices,
  it follows that $n = 3 (k+3)^{(k+1)^2} (k+2) + 1$.
  \begin{eqnarray}
    n         & =    & 3 (k+3)^{(k+1)^2} (k+2) + 1 \nonumber\\
              & \leq & 3 (k+3)^{(k+1)^{2}+1}.       \nonumber \\
    \log{(n)} & \leq & ((k+1)^2 +1) \log{(3 (k+3))}\nonumber \\
              & \leq & (2 k)^2 \log{(3 (k+3))}     \nonumber\\
              & \leq & (2 k)^2 k.                   \nonumber\\
    k         & \geq & \sqrt[3]{\log{(n)}/4}.       \nonumber
  \end{eqnarray}
\end{proof}


Given the result in Theorem~\ref{lower} that no constant number of
trees can collectively 1-span a unit interval graph, it is somewhat
surprising that there is a sparse 1-spanner of an interval graph that
has fewer than $2n - 2$ edges (i.e.,  the number of edges in two
disjoint spanning trees).  In particular, the following theorem is
proved in \cite{CorneilDKY05}.

\begin{theorem}
  Any interval graph $G$ of diameter $D$ admits a sparse additive
  1-spanner with at most $2n - D - 2$ edges.  Moreover, this spanner
  can be constructed in $O(n+m)$ time.
\end{theorem}




\section{The ``Master'' Theorem for Lower Bounds on Collective Tree
  Spanners}\label{master}

We now turn our attention to collective additive spanners on many
other classes of graphs beyond unit interval graphs.  Our main
contribution is a powerful ``master'' theorem that proves the various
lower bounds, by introducing the notion of a \emph{gadget} and then
building graphs that are a \emph{tree of gadgets}.  Each gadget has a
root $r$ and terminals, $t_1, t_2, \cdots, t_j, j \ge 2$.  The tree is
formed by taking one gadget as the ``root gadget'' and then building
the tree by identifying the root node of a child gadget with a
terminal node of its parent.  In this way, we form a complete $j$-ary
tree of gadgets and consider such graphs to be sufficiently large to
establish our lower bounds.  Furthermore, when constructing graph $H$
from gadget $G$, we say that $H$ has \emph{depth} $1$ if $G \equiv H$
and has depth $k \: (k>1)$ if all terminal nodes of the root gadget
are the roots of a graph of depth $k-1$.  The terminals of the gadgets
at depth $k$ are called \emph{$k-$leaves of $H$}.

A spanning tree $T$ of a gadget $G$ is called a \emph{Fast Connecting
  (FC) tree of $G$} if $\dist_T(r,t_i) =
  \dist_G(r,t_i), \: \forall \, 1 \le i \le j$.  Similarly, a spanning
tree $T$ of graph $H$ is called a \emph{Fast Connecting (FC) tree of
  $H$} if the restriction of $T$ to each copy of $G$ in $H$ is an FC
tree of $G$.  For a given $c$, a gadget $G$ is called a
\emph{$d$-gadget} if no $d$ FC trees provide a collection of trees
that collectively $+c$ spans $G$.  In particular, we
require at least one pair of terminals $t_1, t_2$ to witness the fact
that $G$ is not collectively $+c$ spanned (i.e.,
$\dist_T(t_1,t_2) > \dist_G(t_1,t_2) +c, \: \forall
  \, T \in \cal{T}$ where $\cal{T}$ is any arbitrary collection of $d$ FC
trees of $G$).  If the gadget $G$ cannot be collectively $+c$ spanned
by \emph{any} set of FC trees, we refer to the gadget as an
\emph{$\infty$-gadget}.

For example, consider the \emph{house} (see Fig.~\ref{fig:sub:house})
to be the gadget $G$, where we are looking for a collection of trees
that $+2$ spans $G$.  We let $r$ be the \emph{roof of the house},
namely the vertex of degree $2$ that is in the triangle.  The
neighbors of $r$ are labeled $1,2$ and the terminal $t_1$ is the new
neighbor of $1$ and $t_2$ is the new neighbor of $2$.  The only FC
tree is the $P_5: t_1 - 1 - r - 2 - t_2$ and it does not $+2$ span the
house since $\dist_G(t_1,t_2) = 1$ whereas
$\dist_T(t_1,t_2) = 4$.  Thus
for $c=2$, the house is an $\infty$-gadget.

\begin{theorem}\label{main}[The ``Master'' Theorem]
  For a given $c \ge 2$, if $G$ is a $d$-gadget for $d \ge 2$, then no
  $d$ trees can collectively $+c$ span a sufficiently large $H$
  constructed from $G$.
\end{theorem}

\begin{proof}
  The following lemma is key to the proof:

  \begin{lemma}\label{lem:FC-tree-lemma}
    Let $\mathcal{T}$ be a set of $d$ trees that collectively
    spans $H$ of depth $D$.  If all but $k$ trees in
    $\mathcal{T}$ ($1 \leq k \leq d$) are FC trees in $H$ then there
    is a constant $C_k$ such that $D \ngtr C_k$.
  \end{lemma}

  \noindent
  Note: If $d= \infty$, then the lemma says that no constant sized
  collection of trees can $+c$ spans a sufficiently
  large $H$.

  \medskip{}

  \begin{proof}
    For the base case (i.e.~$k=1$) let $T$ be the tree in
    $\mathcal{T}$ that is not an FC tree of $H$.
    We now show that the restriction of $T$ in $H$ must cause every
    copy of $G$ to have at least one root-terminal pair where their
    distance in $T$ is at least one more than their distance in $G$.
    Suppose not, and there exists $G' \in H$ such that $T$ is FC in
    $G'$.  Since $\mathcal{T}$ is a collective $+c$ spanner of $H$,
    $\mathcal{T}$ restricted to $G'$ collectively $+c$ spans $G'$, which
    contradicts $G$ being a $d$-gadget since $G'$ has a collection of
    $d$ FC trees that $+c$ spans $G'$.  For any copy of $G$ in $H$, we
    let $t'$ denote such a terminal and we say that $T$ has a
    \emph{jog} between the root of $G$ and $t'$.

    Now let $t_1$, $t_2$ be terminals of the root gadget of $H$ such
    that none of the FC trees of $H$ in $\mathcal{T}$ $+c$
    spans the pair $t_1$, $t_2$.  Thus only $T$ can
    provide a sufficiently short path between descendants of $t_1$ and
    $t_2$.  Let $a$ (respectively $b$) be a $D$-leaf of $H$ in the
    subgraph of $H$ rooted at $t_1$ (respectively $t_2$) such that the
    path between $t_1$ and $a$ (resp.\ between $t_2$ and $b$) jogs in
    each of the gadgets in the path.  The distance in $T$ between $a$
    and $b$ is at least $2 (D-1)$ greater than their distance in $H$.
    Thus $D \ngtr \lceil \frac{c+2}{2} \rceil$ (i.e., $C_1 = \lceil
    \frac{c+2}{2} \rceil$) as required.

    \medskip{}

    Let us consider the case $k=2$ and let $H$ be a graph of gadgets
    as constructed above, further let $T_1$, $T_2$ be the only non-FC
    trees of $\mathcal{T}$.  By the same argument as above in each of
    the gadgets at least one of $T_1$ and $T_2$ has to be non-FC.  Now
    consider the subgraph $H_{t_1}$ of $H$, that is rooted on vertex
    $t_1$, a terminal vertex of the root gadget.  If $T_1$ is an FC
    tree for all the gadgets in $H_{t_1}$, then, by induction,
    $H_{t_1}$ can not have depth larger than $C_1$.  By the
    construction of $H$ this implies that $D \leq C_1 +1$, which
    proves the lemma.  Thus we can assume that there is some gadget in
    $H_{t_1}$ at level $\ell_1 \le C_1+1$ such that $T_1$ is not FC in
    this gadget.  Let $x$ be the terminal of this gadget that has a
    jog on its $T_1$ path to $t_1$.  Now consider the graph $H_x$
    that is rooted at vertex $x$.  By the same
    argument again, there has to be some gadget in $H_x$ at level
    $\ell_2 \leq 2\cdot C_1 +1$ such that $T_1$ is not FC in this
    gadget.  Applying this argument repeatedly shows that there is a
    vertex $z$ in level $\ell_c \leq c\cdot C_1 + 1$ such that there
    are at least $c$ jogs on the $T_1$ path from $z$ to $t_1$.  Hence
    this path is useless for any pair of vertices $u,v$ where $u$ is
    contained in $H_z$ and $v$ in $H_{t_2}$.  Now we apply the same
    kind of argument in the graph $H_z$ for the tree $T_2$ showing
    that there is a vertex $w$ in level $\ell_{2 \cdot c} \leq 2 \cdot
    c \cdot C_1 +1$ such that any path from $w$ to $t_1$ both in $T_1$
    and in $T_2$ has at least $c$ jogs, implying that $\mathcal{T}$
    can not $+c$ span $H$ if the depth of $H$ is larger than $C_2 := 2
    \cdot c \cdot C_1 +1$.

    \medskip{}

    Now consider case $k$ and suppose for all values $1 \leq s < k$
    the claim holds, i.e., there is such a constant $C_s$ that bounds
    the depth of the graph $H$.  Similarly as in the case $k=2$ we can
    argue that there is a gadget on level $\ell_1 \leq C_{k-1} +1$ for
    which $T_k$ is not an FC-tree and a terminal vertex $x$
    in this gadget such that there is a jog on the $T_k$ path from $x$
    to $t_1$.  Applying this argument repeatedly shows that there is a
    vertex $y$ on level $\ell_c \leq c\cdot C_{k-1} +1$ such that the
    $T_k$ path from $y$ to $t_1$ contains at least $c$ jogs, implying
    that $T_k$ is useless for any pair of vertices $u,v$ with $u$ in
    $H_y$ and $v$ in $H_{t_2}$.  Now we apply the same argument for
    all the other $k-1$ non-FC trees in $\mathcal{T}$ repeatedly,
    showing that there is a terminal vertex $z$ on level
    $\ell_{k \cdot c} \leq k \cdot c \cdot C_{k-1} + 1$ such that in
    each of the trees $T_1, \dots, T_k$ the $z, t_1$ path contains at
    least $c$ jogs implying that these trees are useless for
    connecting vertices from $H_z$ and $H_{t_2}$.  It follows that the
    depth of $H$ can not be larger than $C_k := k \cdot c \cdot
    C_{k-1} + 1$.
  \end{proof}

  By choosing $H$ to be of depth at least $C_d$ and applying the lemma
  we complete the proof of the theorem.
\end{proof}

As an example of the results that can be achieved as a corollary of
the ``master'' theorem consider:

\begin{corollary}\label{house}
  No constant sized set of spanning trees collectively $+2$-spans
  weakly chordal graphs or outerplanar graphs.
\end{corollary}

\begin{proof}
  As we have seen, for $c=2$ the house is an $\infty$-gadget, and thus
  the corollary follows from Theorem \ref{main}.
\end{proof}

As we will show in the next section, this corollary can be
strengthened to hold for any $c \ge 2$.

\section{Lower Bounds on Collective Tree Spanners for Chord\-al and
  Weakly Chordal Graphs}\label{sec:coll-tree-spann}

In this section we construct gadgets to show that for various values
of $c \ge 2$, there are chordal (or weakly chordal)
graphs that cannot have a constant
sized set of spanning trees that collectively $+c$
spans those graphs.  First we strengthen Corollary~\ref{house}.


\begin{theorem}
  For every $c > 2$, no constant number of spanning trees can
  collectively $+c$-span weakly chordal graphs (or outerplanar
  graphs).
\end{theorem}

\begin{proof}
  Consider a $C_4$ $1-2-3-4-1$.  By identifying the $1,2$ vertices of
  one $C_4$ with the $3,4$ vertices of a new $C_4$ we can build a
  chain of $C_4$s such that all vertices except the $1,2$ vertices of
  the first $C_4$ and the $3,4$ vertices of the last $C_4$ have degree
  $3$.  By adding a vertex of degree $2$ adjacent to the $1,2$
  vertices of the first $C_4$ in a chain of $\ell$ $C_4$s, we get an
  \emph{$\ell$-house}.  (Note that a $1$-house is just the house
  mentioned previously.)

  The gadget we use for $c \ge 2$ is a $\lceil \frac{c}{2}
  \rceil$-house with the root being the degree $2$ vertex of the
  triangle and the terminals $t_1,t_2$ being the other degree $2$
  vertices (i.e., the $3,4$ vertices of the last $C_4$).  We note that
  this gadget $G$ is an $\infty$-gadget and that in an FC tree $T$ of
  $G$, $\dist_T(t_1,t_2) = 2 (\lceil \frac{c}{2} \rceil+1)$ whereas
  $\dist_G(t_1,t_2) = 1$.  Now
  using this gadget in Theorem~\ref{main} finishes the proof.
\end{proof}

Note that this result on outerplanar graphs shows that there is a big difference between the collective additive tree $c$-spanners and the low-stretch tree covers.  While, for every $c > 2$, no constant number of spanning trees can collectively $+c$-span outerplanar
graphs, it is known \cite{10353154} that a constant number of trees is sufficient to provide a constant-stretch tree cover even for all minor-free metrics.    

We now turn our attention to chordal graphs.

\begin{theorem}\label{chordal2}
  No constant number of spanning trees can collectively
  $+2$-span chordal graphs.
\end{theorem}


\begin{proof}
  Suppose the theorem is false and that $k$ trees suffice for some
  constant $k$.  (Note that since we don't know of an appropriate
  $\infty$-gadget for chordal graphs, we have to construct a specific
  gadget for each value of $k$.)  In particular, to construct gadget
  $G_k$ form a $K_{k+2}$ on the nodes $r, 1, 2, \cdots, k+1$.  Now add
  vertex $t_0$ adjacent to $1,2, \cdots, k+1$ and vertices
  $t_i, \: 1 \le i \le k+1$ where $t_i$ is adjacent to $t_0$ and $i$.  In effect we
  have identified an edge of a triangle to each $t_0 ~i$ edge; the new
  vertex is $t_i$.  The root of the gadget is $r$ and the terminals
  are the $k+2$ vertices $t_i, \: 0 \le i \le
    k+1$.  See
  Figure~\ref{fig:example-proof-constr} for a depiction of a
  $G_{2}$ gadget (Figure~\ref{fig:sub:G-3})
  as well as a $G_k$ gadget (Figure~\ref{fig:sub:G-k}).
  To prove that $G_k$ is a $k$-gadget
   note that any set of $k$ FC trees of
  $G_k$ must have at least one edge $t_0 t_i, \: 1 \le i \le k+1$ such
  that the distance in all trees between $t_0$ and $t_i$ is $4$.  Now,
  by applying Theorem~\ref{main} and using gadget
  $G_k$ we complete the proof.
\end{proof}

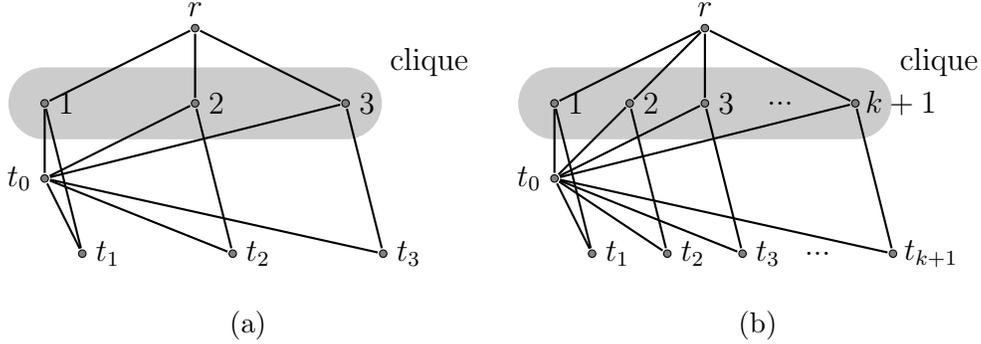
\begin{figure}
  \begin{center}
    \begin{minipage}[t]{0.41\linewidth}
      \begin{tikzpicture}[new set=import nodes,scale=0.5]
  \begin{scope}[nodes={set=import nodes},%
    nodes={%
      inner sep=1pt,draw,circle,draw=black,fill=black!50%
    }, %
    ]%
    \node (r) at (4,8) [label=$r$] {$$}; %
    \node (1) at (0,6) [label=0:$1$] {$$}; %
    \node (3) at (4,6) [label=0:$2$] {$$}; %
    \node (k+1) at (8,6) [label=0:$3$] {$$}; %
    \node (t0) at (0,4) [label=180:$t_0$] {$$}; %
    \node (t1) at (1,2) [label=0:$t_1$] {$$}; %
    \node (t3) at (5,2) [label=0:$t_2$] {$$}; %
    \node (tk+1) at (9,2) [label=0:$t_3$] {$$}; %
  \end{scope}

  \graph[edges={%
    black,shorten >=0.5pt,shorten <=0.5pt,thick %
  }] {%
    (import nodes); %

    r -- {1,3,k+1};%
    t0 -- 1;%
    t0 -- {3,k+1};%
    t0 -- {t1,t3,tk+1};%
    1 -- t1;%
    3 -- t3;%
    k+1 -- tk+1;%
  };%
  \begin{pgfonlayer}{background}
\fill [fill=white] (-1,1) rectangle (11.5,9);
    \node[fill=black!20,rounded rectangle,%
    inner sep=12pt, thick,fit=(1) (3) (k+1),label=5:clique] {};
  \end{pgfonlayer}
\end{tikzpicture} %
      \subcaption{}\label{fig:sub:G-3}
    \end{minipage}
    \begin{minipage}[t]{0.41\linewidth}
      \begin{tikzpicture}[new set=import nodes,scale=0.5]
  \begin{scope}[nodes={set=import nodes},%
    nodes={%
      inner sep=1pt,draw,circle,draw=black,fill=black!50%
    }, %
    ]%
    \node (r) at (4,8) [label=$r$] {$$}; %
    \node (1) at (0,6) [label=0:$1$] {$$}; %
    \node (2) at (2,6) [label=0:$2$] {$$}; %
    \node (3) at (4,6) [label=0:$3$] {$$}; %
    \node (pp1) at (6,6) [draw=black!20,fill=black!20] {...};
    \node (k+1) at (8,6) [label=0:$k+1$] {$$}; %
    \node (t0) at (0,4) [label=180:$t_0$] {$$}; %
    \node (t1) at (1,2) [label=0:$t_1$] {$$}; %
    \node (t2) at (3,2) [label=0:$t_2$] {$$}; %
    \node (t3) at (5,2) [label=0:$t_3$] {$$}; %
    \node (pp1) at (7,2) [draw=white,fill=white] {...};
    \node (tk+1) at (9,2) [label=0:{$t_{k+1}$}] {$$}; %
  \end{scope}

  \graph[edges={%
    black,shorten >=0.5pt,shorten <=0.5pt,thick %
  }] {%
    (import nodes); %

    r -- {1,2,3,k+1};%
    t0 -- 1;%
    t0 -- {2,3,k+1};%
    t0 -- {t1,t2,t3,tk+1};%
    1 -- t1;%
    2 -- t2;%
    3 -- t3;%
    k+1 -- tk+1;%
  };%
  \begin{pgfonlayer}{background}
\fill [fill=white] (-1,1) rectangle (11.5,9);
    \node[fill=black!20,rounded rectangle,%
    inner sep=12pt, thick,fit=(1) (2) (3) (k+1),label=5:clique] {};
  \end{pgfonlayer}
\end{tikzpicture}%
      \subcaption{}\label{fig:sub:G-k}
    \end{minipage}
  \end{center}
  \caption {Gadgets for the proof of Theorem \ref{chordal2}.}
  \label{fig:example-proof-constr}
\end{figure}

\begin{theorem}\label{chordal3}
  No constant number of spanning trees can collectively
  $+3$-span chordal graphs.
\end{theorem}

\begin{proof}
  As with the preceding proof we know of no appropriate
  $\infty$-gadget and have to construct a specific gadget for each
  value of $k$.  In creating these gadgets we will make use of a
  ``3-sun'' consisting of a triangle where every edge of the triangle
  belongs to another unique triangle.

  First we describe the $k=2$ gadget, $G'_2$.  As above we start with
  a $K_4$ on the vertex set $\{ r, 1, 2, 3\}$ and add vertex $t_0$
  universal to $\{1, 2, 3\}$.  But now, instead of identifying an edge
  of a new triangle to each $t_0 ~i$ edge, we identify an exterior
  edge of a new 3-sun with $t_0$ being identified with the degree 4
  vertex of each 3-sun; see Figure~\ref{fig:idea-gadget-proof}.  We
  designate the two new degree 2 vertices of the 3-sun built off the
  $t_0 ~i$ edge as $t'_i$ and ${t''_i}$ and their common neighbor
  as $t_i$.  The terminals of the gadget $G'_2$ are $\{t_1, t'_1,
  {t''_1} , t_2, t'_2, {t''_2} , t_3, t'_3, {t''_3} \}$.  See
  Figure~\ref{fig:idea-gadget-proof}.

  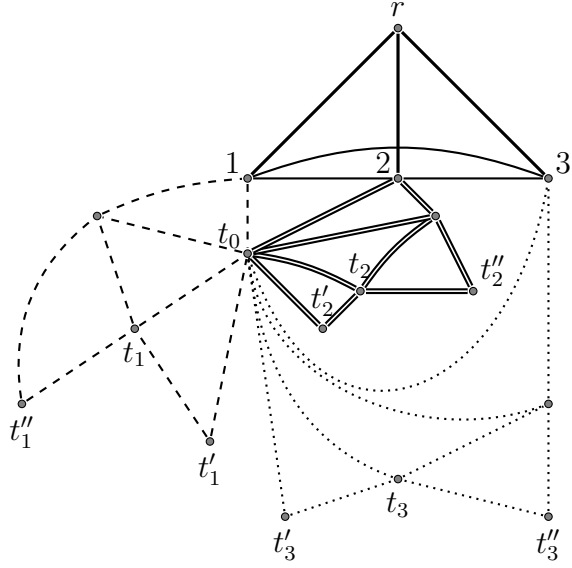
\begin{figure}
    \begin{center}
      \begin{center}
    \begin{tikzpicture}[new set=import nodes]
  \begin{scope}[nodes={set=import nodes},%
    nodes={%
      inner sep=1pt,draw,circle,draw=black,fill=black!50%
    }, %
    ]%
    ]%
    \node (r) at (12,17) [label=$r$] {$$}; %
    \node (1) at (10,15) [label=130:$1$] {$$}; %
    \node (2) at (12,15) [label=130:$2$] {$$}; %
    \node (3) at (14,15) [label=50:$3$] {$$}; %
    \node (t0) at (10,14) [label=135:$t_0$] {$$}; %
    \node (t1) at (8.5,13) [label=270:$t_1$] {$$}; %
    \node (t1'') at (7,12) [label=270:$t_1''$] {$$}; %
    \node (t1') at (9.5,11.5) [label=270:$t_1'$] {$$}; %
    \node (a1) at (8,14.5) [label=135:$$] {$$}; %
    \node (t2) at (11.5,13.5) [label=90:$t_2$] {$$}; %
    \node (t2') at (11,13) [label=90:$t_2'$] {$$}; %
    \node (t2'') at (13,13.5) [label=50:$t_2''$] {$$}; %
    \node (a2) at (12.5,14.5) [label=30:$$] {$$}; %
    \node (t3) at (12,11) [label=270:$t_3$] {$$}; %
    \node (t3') at (10.5,10.5) [label=270:$t_3'$] {$$}; %
    \node (t3'') at (14,10.5) [label=270:$t_3''$] {$$}; %
    \node (a3) at (14,12) [label=0:$$] {$$}; %
  \end{scope}

  \graph[edges={%
    black,shorten >=0.5pt,shorten <=0.5pt,thick %
  }] {%
    (import nodes); %

    r -- [very thick] {1,2,3};%
    1 -- [] 2 -- 3;%
    1 -- [bend left=20] 3;%

t1' -- [dashed] t0;
t1'' -- [dashed] t1;
t1' -- [dashed] t1;
t1 -- [dashed] a1;
t1 -- [dashed] t0;
t1'' -- [dashed, bend left] a1;
t0 -- [dashed] a1;
a1 -- [dashed,bend left=10] 1;
1 -- [dashed] t0;

t2' -- [double] t0;
t2' -- [double] t2;
t2'' -- [double] t2;
t2 -- [double, bend left=10] a2;
t2 -- [double, bend right=10] t0;
t2'' -- [double] a2;
t0 -- [double] a2;
a2 -- [double] 2;
2 -- [double] t0;

t3' -- [dotted] t0;
t3' -- [dotted] t3;
t3'' -- [dotted] t3;
t3 -- [dotted] a3;
t3 -- [dotted, out=160, in=277] t0;
t3'' -- [dotted] a3;
t0 -- [dotted, out=277, in=200] a3;
a3 -- [dotted] 3;
  };%
\draw[thick,dotted]  (t0) .. controls (11,11) and (13.5,12) .. (3);
\end{tikzpicture}
      \end{center}
    \end{center}
    \caption{Idea of the gadget in the proof of
      Theorem\ref{chordal3}.}
    \label{fig:idea-gadget-proof}
  \end{figure}
  As above, any two FC trees of $G_2$ must leave at least one of the
  edges, $\{t_0 1, t_0 2, t_0 3\}$ appearing in neither of the FC
  trees; without loss of generality assume edge $t_0 1$ does not
  appear.  Now in both of the FC trees either edge $t_1 t'_1$ or edge
  $t_1 {t''_1}$ must have distance $6$ in the FC tree.

  The generalization of gadget $G'_2$ to $G'_k$ is straightforward; we
  start with the $G_k$ gadget introduced in the proof of
  Theorem~\ref{chordal2} and identify
  an exterior edge of a new 3-sun with the $t_0 i$
    edge where $t_0$ is identified with the degree $4$ vertex
  of each 3-sun.  We designate the two new degree $2$ vertices of the
  3-sun built off the $t_0 i$ edge as
  $t'_i$ and ${t''_i}$ and their common neighbour as $t_i$.  The
  terminals of the gadget $G'_2$ are $\{t_i, t'_i, {t''_i} : 1 \le i
  \le k+1\}$.  Again there is some edge $t_0 i: 1 \le i \le k+1$ that
  appears in none of the $k$ FC trees of $G'_k$ and in all such trees
  the edge $t_i t'_i$ or the edge
  $t_i {t''_i}$ must have distance $6$ in the FC tree.

  Thus by Theorem~\ref{main} and by using gadget
  $G'_k$ we complete the proof.
\end{proof}



\section{Lower Bound of Collective Tree Spanners for Strongly Chordal
  Graphs}\label{sec:lower-bound-coll}


As mentioned in the Introduction, one spanning tree
is sufficient to +3-span any strongly chordal graph
\cite{BrandstadtCD99} and at most $\log_2n$ spanning trees are
sufficient to +2-span every chordal graph \cite{DYL}. The following
lower bound shows that these results are tight.

\begin{theorem}\label{thm:strchrnoconstant}
  No constant number of spanning trees can collectively $+2$-span
  strongly chordal graphs.
\end{theorem}

We create a family of strongly chordal graphs that will serve as
gadgets for showing that no constant number of trees can collectively
$+2$ span strongly chordal graphs.  More precisely, for each positive
integer $\delta$ we define a strongly chordal graph $G_{f(\delta)}$
(where $f(\delta)$ is a function of $\delta$) that serves as a gadget
for the general construction in the ``Master'' Theorem
(Thm.~\ref{main}).  Note that, as in the
case of chordal graphs, we need a different
gadget for every $\delta$.

First we will define a family of strongly chordal graphs $G_k$ ($k
\geq 1$).  Then we will show that for each positive integer $\delta$
there is a graph $G_{f(\delta)}$, such that no set of $\delta$ FC trees
can collectively $+2$ span the graph
$G_{f(\delta)}$.  Finally, by using
Theorem~\ref{main}, we can conclude that for each $\delta$ there is a
strongly chordal graph such that there is no set of $\delta$ trees
that collectively $+2$ spans this graph.

\paragraph{Definition of $G_k$.}
\label{sec:definition-g_d}

The definition of the graphs $G_k$ is somewhat involved.  Therefore,
for better understanding of the construction, the graphs $G_3$ and
$G_4$ are depicted in Figure~\ref{fig:example-gadget-case}
and~\ref{fig:larger-example}.  The graph $G_k$ is defined to consist
of a root vertex $r$ and three sets of vertices, the $b$ vertices, the
$c$ vertices and the $d$ vertices.  There are $2^k$ $b$ vertices
$\{b_0, b_1, \dots, b_{2^k-1}\}$; these vertices form a clique and are
all adjacent to the root vertex $r$.  Similarly, there are $2^k$ $d$
vertices $\{d_0, d_1, \dots, d_{2^k-1}\}$ and also these vertices form
a clique.  Furthermore all the edges between all $b$ and $d$ vertices
are present, i.e., the $b$ and $d$ vertices together form a clique
with $2\cdot 2^k$ vertices.  The definition of the $c$ vertices and
their adjacencies is a bit more complex.  For this consider a
recursive partition of the set of $d$ vertices.  Let $D^{\emptyset}$
be the set of all $d$ vertices, i.e., $D^{\emptyset} = \{d_0, d_1,
\dots, d_{{2^k}-1}\}$.  Now partition $D^{\emptyset}$ into two sets
(the ``left'' and the ``right'' vertices of $D^{\emptyset}$) $D^{\ell}
=\{d_0, d_2, \dots, d_{2^{k-1}-1}\}$ and $D^{r} = \{d_{2^{k-1}},
\dots, d_{2^{k}-1}\}$.  In the next step partition each of the sets
$D^{\ell}$ and $D^r$ again into their left and right subsets (denote
them by adding ``$\ell$'' or ``r'' to their superscript,
correspondingly) and repeat this partitioning recursively, until each
set contains only one element.
For the case $k=3$ this means:\\
$D^{\emptyset}=\{d_0, d_1, d_2, d_3, d_4, d_5,
d_6, d_7\}$,\\
$D^{\ell}=\{d_0, d_1, d_2, d_3\}$, $D^{r}=\{d_4, d_5, d_6,
d_7\}$,\\
$D^{\ell \ell}=\{d_0, d_1\}$, $D^{\ell r}=\{d_2, d_3\}$,
$D^{r \ell}=\{d_4, d_5\}$, $D^{rr}=\{d_6, d_7\}$,\\
$D^{\ell \ell \ell}=\{d_0\}$, $D^{\ell \ell r}=\{d_1\}$, $D^{\ell r
  \ell}=\{d_2\}$, $D^{\ell r r}=\{d_3\}$, $D^{r \ell \ell}=\{d_4\}$,
$D^{r \ell r}=\{d_5\}$, $D^{r r \ell}=\{d_6\}$, $D^{r r r}=\{d_7\}$.

For each of these $D$ sets we create a $c$ vertex with the same index,
that is adjacent to all vertices of this particular $D$ set.  In our
example, for the set $D^{\ell r}=\{d_2, d_3\}$ we create a vertex
$c^{\ell r}$ that is adjacent to $d_2$ and $d_3$.  Furthermore, if
there is an edge between a $c$ vertex and some $d$ vertex $d_i$ we add
the corresponding edge between that $c$ vertex and the vertex $b_i$.
In our example, the vertex $c^{\ell r}$ is also adjacent to the
vertices $b_2$ and $b_3$.  Since the vertices $b_i$ and $d_i$ have the
same neighborhood among the $c$ vertices we will later refer to them
as ``twins''.  Because of the hierarchical structure of the collection
of $D$ sets the $c$ vertices can be assigned to so-called layers,
where layer~$0$ consists only of vertex $c^{\emptyset}$, layer~$1$
contains the vertices $c^{\ell}$ and $c^{r}$, layer~$2$ the vertices
$c^{\ell \ell}, c^{\ell r}, c^{r \ell}, c^{r r}$ and, in general,
layer~$t$ contains all vertices with an $\ell/r$ sequence of length
$t$.  The graph $G_k$ is also called a \emph{star graph}
in~\cite{stargraph}.

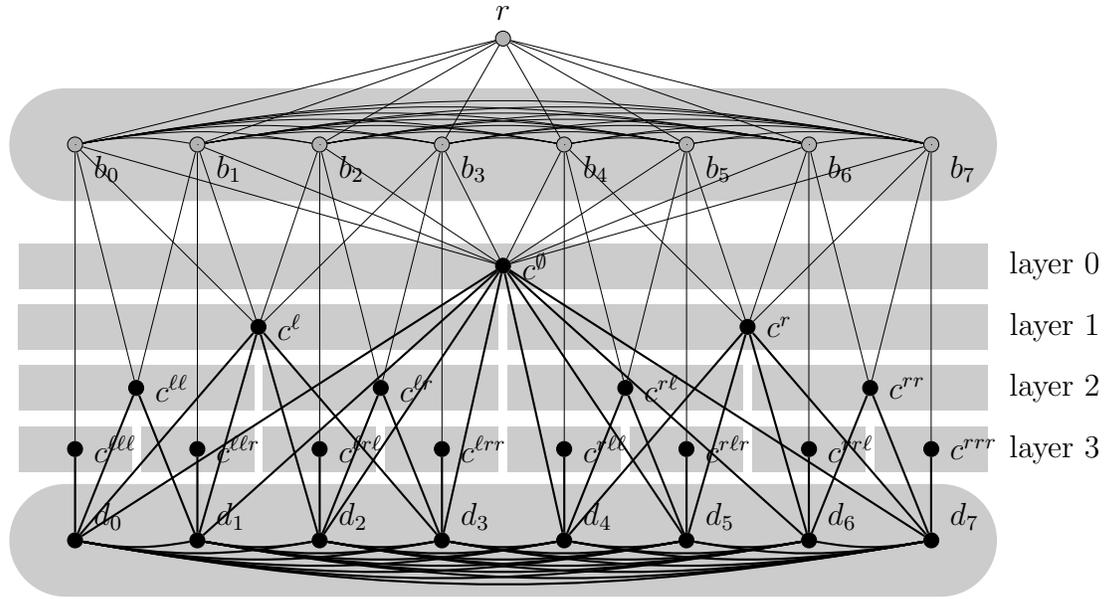
\begin{figure}
\tikzset{%
  regularNode/.style={%
    inner sep=2pt,draw,circle,draw=black,fill=black!30%
  },%
  terminalNode/.style={%
    inner sep=2pt,draw,circle,draw=black,fill=black%
  },%
  regularEdge/.style={%
    very thin%
  },%
  terminalEdge/.style={%
    thick%
  },%
  nodeBLevel/.style={
    fill=black!20,%
    rounded rectangle,%
    inner sep=1pt,%
    minimum height=15mm,%
    minimum width=132mm%
  },%
  nodeLevel0/.style={
    fill=black!20,%
    rectangle,%
    minimum height=6mm,%
    minimum width=128.8mm%
  },%
  nodeLevel1/.style={
    fill=black!20,%
    rectangle,%
    minimum height=6mm,%
    minimum width=63.8mm%
  },%
  nodeLevel2/.style={
    fill=black!20,%
    rectangle,%
    minimum height=6mm,%
    minimum width=31.25mm%
  },%
  nodeLevel3/.style={
    fill=black!20,%
    rectangle,%
    minimum height=6mm,%
    minimum width=15mm%
  },%
  nodeDLevel/.style={
    fill=black!20,%
    rounded rectangle,%
    inner sep=1pt,%
    minimum height=15mm,%
    minimum width=132mm%
  }%
}%
\begin{center}
\begin{tikzpicture}
  \matrix[row sep=2mm,column sep=6mm] {%

    & & & & & & & \node[regularNode] (root) {};\\ \\ \\ \\ \\ \\%

    \node[regularNode] (b0) {}; & &%
    \node[regularNode] (b1) {}; & &%
    \node[regularNode] (b2) {}; & &%
    \node[regularNode] (b3) {}; & &%
    \node[regularNode] (b4) {}; & &%
    \node[regularNode] (b5) {}; & &%
    \node[regularNode] (b6) {}; & &%
    \node[regularNode] (b7) {};\\%

    \\ \\ \\ \\ \\ \\%

    & & & & & & & %
    \node[terminalNode] (c0p0) {};%
    \\ \\ \\%
    & & & %
    \node[terminalNode] (c1p0) {}; & & & & & & & & %
    \node[terminalNode] (c1p1) {};\\ \\ \\%
    & %
    \node[terminalNode] (c2p0) {}; & & & & %
    \node[terminalNode] (c2p1) {}; & & & & %
    \node[terminalNode] (c2p2) {}; & & & & %
    \node[terminalNode] (c2p3) {};\\ \\ \\%
    \node[terminalNode] (c3p0) {}; & & %
    \node[terminalNode] (c3p1) {}; & & %
    \node[terminalNode] (c3p2) {}; & & %
    \node[terminalNode] (c3p3) {}; & & %
    \node[terminalNode] (c3p4) {}; & & %
    \node[terminalNode] (c3p5) {}; & & %
    \node[terminalNode] (c3p6) {}; & & %
    \node[terminalNode] (c3p7) {}; \\%
    \\ \\ \\ \\ %

    \node[terminalNode] (d0) {}; & &%
    \node[terminalNode] (d1) {}; & &%
    \node[terminalNode] (d2) {}; & &%
    \node[terminalNode] (d3) {}; & &%
    \node[terminalNode] (d4) {}; & &%
    \node[terminalNode] (d5) {}; & &%
    \node[terminalNode] (d6) {}; & &%
    \node[terminalNode] (d7) {};\\%
  };%
  \graph {
    (root) -- [regularEdge] %
    {(b0), (b1), (b2), (b3), (b4), (b5), (b6), (b7)};%

    (c0p0) -- [regularEdge] {(b0), (b1), (b2), (b3), (b4), (b5), (b6), (b7)};%
    (c0p0) -- [terminalEdge] {(d0), (d1), (d2), (d3), (d4), (d5), (d6), (d7)};%

    (c1p0) -- [regularEdge] {(b0), (b1), (b2), (b3)};%
    (c1p0) -- [terminalEdge] {(d0), (d1), (d2), (d3)};%

    (c1p1) -- [regularEdge] {(b4), (b5), (b6), (b7)};%
    (c1p1) -- [terminalEdge] {(d4), (d5), (d6), (d7)};%

    (c2p0) -- [regularEdge] {(b0), (b1)};%
    (c2p0) -- [terminalEdge] {(d0), (d1)};%

    (c2p1) -- [regularEdge] {(b2), (b3)};%
    (c2p1) -- [terminalEdge] {(d2), (d3)};%

    (c2p2) -- [regularEdge] {(b4), (b5)};%
    (c2p2) -- [terminalEdge] {(d4), (d5)};%

    (c2p3) -- [regularEdge] {(b6), (b7)};%
    (c2p3) -- [terminalEdge] {(d6), (d7)};%

    (c3p0) -- [regularEdge] {(b0)};%
    (c3p0) -- [terminalEdge] {(d0)};%

    (c3p1) -- [regularEdge] {(b1)};%
    (c3p1) -- [terminalEdge] {(d1)};%

    (c3p2) -- [regularEdge] {(b2)};%
    (c3p2) -- [terminalEdge] {(d2)};%

    (c3p3) -- [regularEdge] {(b3)};%
    (c3p3) -- [terminalEdge] {(d3)};%

    (c3p4) -- [regularEdge] {(b4)};%
    (c3p4) -- [terminalEdge] {(d4)};%

    (c3p5) -- [regularEdge] {(b5)};%
    (c3p5) -- [terminalEdge] {(d5)};%

    (c3p6) -- [regularEdge] {(b6)};%
    (c3p6) -- [terminalEdge] {(d6)};%

    (c3p7) -- [regularEdge] {(b7)};%
    (c3p7) -- [terminalEdge] {(d7)};%

  };%

  \foreach \i in {0, ..., 7} {%
    \foreach \j in {\i, ..., 7} { %
      \draw [-,regularEdge] (b\i) to [bend left=10] (b\j);%
    }%
  }%

  \foreach \i in {0, ..., 7} {%
    \foreach \j in {\i, ..., 7} { %
      \draw [-,terminalEdge] (d\i) to [bend right=10] (d\j);%
    }%
  }%


  \begin{pgfonlayer}{background}
    \node[nodeBLevel,fit=(b0) (b1) (b2) (b3) (b4) (b5) (b6) (b7)] {};%
    \node[nodeLevel0,fit=(c0p0),label=0:~~layer $0$] {};%
    \node[nodeLevel1,fit=(c1p0)] {};%
    \node[nodeLevel1,fit=(c1p1),label=0:~~layer $1$] {};%
    \node[nodeLevel2,fit=(c2p0)] {};%
    \node[nodeLevel2,fit=(c2p1)] {};%
    \node[nodeLevel2,fit=(c2p2)] {};%
    \node[nodeLevel2,fit=(c2p3),label=0:~~layer $2$] {};%
    \node[nodeLevel3,fit=(c3p0)] {};%
    \node[nodeLevel3,fit=(c3p1)] {};%
    \node[nodeLevel3,fit=(c3p2)] {};%
    \node[nodeLevel3,fit=(c3p3)] {};%
    \node[nodeLevel3,fit=(c3p4)] {};%
    \node[nodeLevel3,fit=(c3p5)] {};%
    \node[nodeLevel3,fit=(c3p6)] {};%
    \node[nodeLevel3,fit=(c3p7),label=0:~~layer $3$] {};%
    \node[nodeDLevel,fit=(d0) (d1) (d2) (d3) (d4) (d5) (d6) (d7)] {};%
  \end{pgfonlayer}

  \coordinate [label=$r$] (rootLabel) at (root.north);

  \coordinate [label=315:$b_0$] (b0Label) at (b0.east); %
  \coordinate [label=315:$b_1$] (b1Label) at (b1.east); %
  \coordinate [label=315:$b_2$] (b2Label) at (b2.east); %
  \coordinate [label=315:$b_3$] (b3Label) at (b3.east); %
  \coordinate [label=315:$b_4$] (b4Label) at (b4.east); %
  \coordinate [label=315:$b_5$] (b5Label) at (b5.east); %
  \coordinate [label=315:$b_6$] (b6Label) at (b6.east); %
  \coordinate [label=315:$b_7$] (b7Label) at (b7.east); %

  \coordinate [label=0:$c^{\emptyset}$] (c0p0Label) at (c0p0.east);%
  \coordinate [label=0:$c^{\ell}$] (c1p0Label) at (c1p0.east); %
  \coordinate [label=0:$c^{r}$] (c1p1Label) at (c1p1.east); %
  \coordinate [label=0:$c^{\ell \ell}$] (c2p0Label) at (c2p0.east); %
  \coordinate [label=0:$c^{\ell r}$] (c2p1Label) at (c2p1.east); %
  \coordinate [label=0:$c^{r \ell}$] (c2p2Label) at (c2p2.east); %
  \coordinate [label=0:$c^{rr}$] (c2p3Label) at (c2p3.east); %
  \coordinate [label=0:$c^{\ell \ell \ell}$] (c3p0Label) at (c3p0.east); %
  \coordinate [label=0:$c^{\ell \ell r}$] (c3p1Label) at (c3p1.east); %
  \coordinate [label=0:$c^{\ell r \ell}$] (c3p2Label) at (c3p2.east); %
  \coordinate [label=0:$c^{\ell r r}$] (c3p3Label) at (c3p3.east); %
  \coordinate [label=0:$c^{r \ell \ell}$] (c3p4Label) at (c3p4.east); %
  \coordinate [label=0:$c^{r \ell r}$] (c3p5Label) at (c3p5.east); %
  \coordinate [label=0:$c^{r r \ell}$] (c3p6Label) at (c3p6.east); %
  \coordinate [label=0:$c^{r r r}$] (c3p7Label) at (c3p7.east); %

  \coordinate [label=45:$d_0$] (d0Label) at (d0.east); %
  \coordinate [label=45:$d_1$] (d1Label) at (d1.east); %
  \coordinate [label=45:$d_2$] (d2Label) at (d2.east); %
  \coordinate [label=45:$d_3$] (d3Label) at (d3.east); %
  \coordinate [label=45:$d_4$] (d4Label) at (d4.east); %
  \coordinate [label=45:$d_5$] (d5Label) at (d5.east); %
  \coordinate [label=45:$d_6$] (d6Label) at (d6.east); %
  \coordinate [label=45:$d_7$] (d7Label) at (d7.east); %
\end{tikzpicture}
\caption{Example of gadget for the case $k=3$.  In addition to the
  edges depicted in the figure, all edges between the $b$ and the $d$
  vertices exist.  Terminal vertices are black and terminal edges are
  thick.  The gray boxes around the $c$ vertices are to point out the
  adjacency of the corresponding $c$ vertex to the $b$ and $d$
  vertices.}
  \label{fig:example-gadget-case}
\end{center}
\end{figure}

\begin{figure}
\input{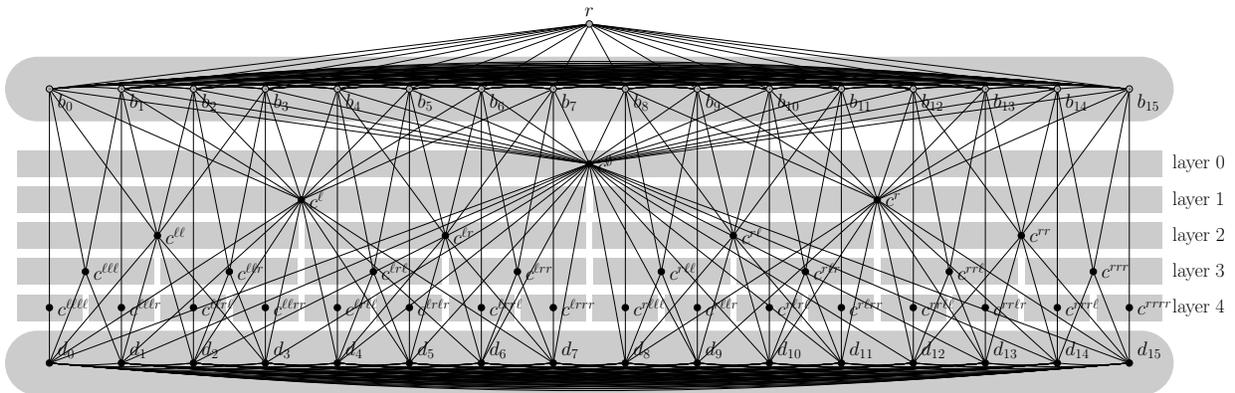}
\caption{Gadget for the case $k=4$.}
  \label{fig:larger-example}
\end{figure}

\begin{lemma}\label{lemma:strongly chordal}
  The graph $G_k$ is strongly chordal.
\end{lemma}

\begin{proof}
  To prove that $G_k$ is strongly chordal, we will use the
  characterization of strongly chordal graphs provided by Farber in
  ~\cite{faber-strongly-chordal-83}; namely, a graph is strongly
  chordal if and only if it admits a vertex ordering
  $\sigma=[v_1,v_2,\dots, v_n]$, (called a {\em simple elimination
    ordering}), such that vertex $v_i$ is simple in $G[\{v_i, \dots,
  v_n\}]$, where a vertex $v$ is called \emph{simple} if the closed
  neighborhoods of $v$'s neighbors can be linearly ordered by set
  inclusion, or, alternatively, if for each pair of neighbors $u,w$ of
  $v$ we have either $N[u] \subseteq N[w]$ or $N[w] \subseteq N[u]$.
  In particular, we need to show that there is a
  simple elimination ordering $\sigma$ of the vertices of $G_k$.
  Consider the following ordering:



  \begin{itemize}
  \item Start with the set of $c$ vertices and remove them beginning
    with the largest numbered layer (i.e., the lowest layer in
    Figures~\ref{fig:example-gadget-case}
    and~\ref{fig:larger-example}, iteratively up to layer~0.  Each of
    these $c$ vertices is only adjacent to $b$ and $d$ vertices and,
    since the $c$ vertices are removed beginning with the largest
    numbered layer layer, all these $d$
    vertices have the same neighborhood and all these $b$ vertices
    also have the same neighborhood plus vertex $r$.
  \item Next remove all $d$ vertices.  They form a clique with the $b$
    vertices and have no further neighbors.
  \item Finally remove all $b$ vertices, followed by vertex $r$.
  \end{itemize}
\end{proof}

\medskip{}

For the construction of FC trees on the graphs $G_k$ we need to
consider the distance levels from the root node $r$
(note that these \emph{levels} should not be confused
  with the above defined \emph{layers}).  It is easy to see that the
$b$ vertices are at distance $1$ from $r$ and the $c$ and $d$ vertices
are at distance $2$ from $r$.  All vertices on the distance-$2$ level
are considered to be \emph{terminals} of the gadget.  Thus a spanning
tree $T$ of the gadget $G_k$ will be a Fast
Connecting (FC) tree if $\dist_T(r,x) =
\dist_{G_k}(r,x)$, where $x$ is an arbitrary vertex on level $2$.
Therefore the edges on level 2 (called \textit{terminal edges}) cannot
belong to any FC tree.

\medskip

Now we will show that for each positive integer $\delta$ there is a
gadget graph $G_k$ with $k = f(\delta)$ such that there are no
$\delta$ FC trees that collectively $+2$ span $G_k$.  According to the
definition of FC trees, a vertex on level $2$ of an FC tree must have
a parent vertex in that FC tree on level $1$, and a vertex of an FC
tree on level $1$ must have $r$ as its parent.  Thus, it is easy to
observe that an FC tree $T$ can either $+1$ span or $+3$ span a
terminal edge.  Therefore, for showing that there are no $\delta$ FC
trees that collectively $+2$ span $G_k$, it is
sufficient to prove that no set of
$\delta$ FC trees can collectively $+1$ span all terminal edges of
$G_k$.

\medskip{}

For the beginning let $\delta=2$ and consider the graph $G_k$ with $k
= \delta \cdot (\lceil \log_2{\delta+1} \rceil)=4$.  Suppose there is
a set of $\delta=2$ FC trees $\{T^1, T^2\}$ that collectively $+1$
span $G_4$ (see Figure~\ref{fig:larger-example}).  We know that each
edge between $c^{\emptyset}$ and any of the $d$ vertices, say $d_i$,
has to be covered in one of $T^1$ or $T^2$, where \emph{cover} means
that $c^{\emptyset}$ and $d_i$ have 
either a common neighbor $b^1$ in $T^1$ and/or a common neighbor $b^2$
in $T^2$.  Since $T^1$ and $T^2$ are trees there is only one such
vertex $b^1$ in $T^1$ and one such vertex $b^2$ in $T^2$ (if all
$c^{\emptyset} d_i$ edges are covered in only one of the trees, say
$T^1$ then select an arbitrary $b$ vertex as $b^2$).

Later we will make use of the fact that each $d$ vertex is adjacent
either to $b^1$ in $T^1$ or to $b^2$ in $T^2$.

As noted earlier, the vertices $b^1$ and $b^2$ each have a ``twin''
$d^1$ and $d^2$ in the set of $d$ vertices, that has the same
neighborhood among the $c$ vertices.  Now consider the hierarchy of
$D$ sets.  By the definition of these sets it is clear that there is
only one set in each layer that contains the vertex $d^1$ and,
similarly, there is only one set in each layer that contains the
vertex $d^2$.  Consequently, each of $b^1$ and $b^2$ has only one
neighbor on each of the layers of the $c$ vertices.  Hence, among the
four $c$ vertices on layer~$2$ there are
two that are neither adjacent to $b^1$ nor to $b^2$.  Let $c^{1}$ be
either of these two $c$ vertices on layer~$2$ and $D^{1}$ the
corresponding $D$ set.

Similarly as in the initial step, for all $d_i \in D^{1}$ the edge
$c^{1} d_i$ has to be covered in at least one of the trees $T^1$ or
$T^2$.  Since both trees are FC trees and $c^{1}$ is neither adjacent
to $b^1$ nor to $b^2$, there have to be vertices ${b^1}', {b^2}'$ such
that the edge ${b^1}' c^{1}$ is contained in $T^1$ and/or
the edge ${b^2}' c^{1}$ is contained in $T^2$.

Next observe, that if the edge $c^{\emptyset} d_i$ is covered in
$T^1$, then the edge $c^{1} d_i$ cannot be covered in $T^1$ because
otherwise the vertices $r, b^1, d_i, {b^1}'$ form a cycle in $T^1$,
contradicting $T^1$ being a tree.  Consequently, for all vertices $d_i
\in D^{1}$ whose edge with $c^{\emptyset}$ is covered in $T^1$, the
edge $c^{1} d_i$ has to be covered in $T^2$; reversely, for all
vertices $d_i$ whose $c^{\emptyset} d_i$ edge is covered in $T^2$, the
edge $c^{1} d_i$ has to be covered in $T^1$.  (***) Note that now each
$d_i \in D^{1}$ is adjacent to exactly one of $b^1, {b^1}'$ in $T^1$
and is also adjacent to exactly one of $b^2, {b^2}'$ in $T^2$.

Now consider the four $c$-vertices on layer~$4$
($4=2+2=1+\lceil \log_2{d} \rceil + 1+\lceil \log_2{d} \rceil$) that
are ``below'' $c^{1}$ (i.e., the union of their neighborhood in the
$D$ layer forms exactly the set $D^{1}$).  None of these vertices is
adjacent to $b^1$ or to $b^2$ and at least two of them are neither
adjacent to ${b^1}'$ nor to ${b^2}'$.  Let $c^{2}$ be one of these two
vertices, let $D^{2}$ be the corresponding $D$ set and consider a
vertex $d_i$ in $D^{2}$ (in fact there is only one such vertex).
Again the edge $c^{2} d_i$ needs to be covered and, since $c^{2}$ is
not adjacent to any of $b^1, {b^1}', b^2, {b^2}'$, there has to be a
$b''$ vertex connecting $c^{2}$ with $r$, that covers the $c^{2} d_i$
edge in one of the trees, say in $T^1$.  However, then we have a cycle
on $d_i, b'', r$ and $b^1$ or ${b^1}'$, because of (***),
contradicting $T^1$ being a tree.

Consequently, there is no set of two FC trees that can collectively +2
span all edges of the gadget.

\bigskip{}

Based on the construction for $\delta=2$ the proof for arbitrary
$\delta$ can be formulated easily.  The important observation is that
each time that a new $c$ vertex is considered (in the case $\delta=2$
these were vertices $c^{\emptyset}$, $c^{1}$, $c^{2}$) for every $d$
vertex of the corresponding nested sequence of $D$ sets
($D^{\emptyset}, D^{1}, D^{2}$) the parent $b$ vertex in at least one
of the $\delta$ FC trees is fixed.  Thus after $\delta$ rounds of
selecting $c$ vertices, there is a $d$ vertex, say $\tilde{d}$ for
which in each of the $\delta$ FC trees the parent $b$ vertex is
already fixed.  If now a suitable $c$ vertex, say $\tilde{c}$, is
selected that is not adjacent to any of the already fixed $b$
vertices, then the $\tilde{c} \tilde{d}$ edge can not be covered in
any of the $\delta$ FC trees.

More precisely, the argument is as follows.  Let $\delta$ be an
integer larger than $2$ and consider the graph $G_k$ with $k = \delta
\cdot (\lceil \log_2{\delta+1} \rceil)$.  Suppose for contradiction,
there is a set of $\delta$ FC trees $\{T^1, \dots, T^{\delta}\}$ that
collectively $+1$ span $G_k$.  Now consider vertex $c^{\emptyset}$; in
each of the $\delta$ FC trees there is a $b$ parent of
$c^{\emptyset}$.  Further, each edge between $c^{\emptyset}$ and a
vertex in $D^{\emptyset}$ has to be covered in one of the $\delta$
trees and, hence, by one of those $b$ vertices; we call the set of
these $\delta$ $b$ vertices $B^{\emptyset}$.  Consider the successors
of vertex $c^{\emptyset}$ in the hierarchy of $c$ vertices on
layer~$\lceil \log_2{\delta+1} \rceil$;
since there are more than $\delta + 1$ successor vertices in that
layer, there has to be at least one such
$c$ vertex, say $c^{1}$, that is not adjacent to any vertex in
$B^{\emptyset}$; let $D^{1}$ be the corresponding $D$ set of $c^{1}$.
Again, $c^{1}$ has to have a $b$ parent in each of the FC trees (we
call this set $B^{1}$) and the edges from $c^{1}$ to each of the
vertices of $D^{1}$ have to be covered by these $B^{1}$ vertices,
since all $b$ vertices in $B^{\emptyset}$ are not adjacent to $c^{1}$.
Note further, that each of the $D^{1}$ vertices already has a $b$
parent in $B^{\emptyset}$ since its edge to $c^{\emptyset}$ had to be
covered as well by an FC tree.  Now we continue by
selecting a vertex $c^{2}$ that is a successor of $c^{1}$ in the
hierarchy of $c$ vertices in layer~$2 \cdot
\lceil \log_2{\delta+1} \rceil$ and that is not adjacent to any vertex
in $B^{1}$.  By the hierarchy of the $c$ vertices $c^{2}$ is not
adjacent to any vertex in $B^{\emptyset}$ either.  Again we find a set
of $b$ vertices $B^{1}$ that cover all the
edges between $c^{2}$ and the vertices in $D^{2}$.

We continue in this fashion and finally select a vertex $c^{\delta}$
on layer $\delta \cdot \lceil \log_2{\delta+1} \rceil$, that is not
adjacent to any of the previously selected vertices in the $B^{i}$
sets ($i \in \{\emptyset, 1, \dots, \delta-1\}$).  Furthermore, each
vertex in $D^{\delta}$ already has a $b$ parent in each of the sets
$B^{i}$.  To cover the edges from $c^{\delta}$ to such a vertex $d \in
D^{\delta}$ we need to have a common $b$ parent of $d $ and
$c^{\delta}$ in one of the $\delta$ FC trees.  However, vertex $d$
already has a $b$ parent in each of these trees and, by the definition
of $c^{\delta}$, none of these $b$ vertices is adjacent to
$c^{\delta}$.  Consequently, none of the terminal edges between
$c^{\delta}$ and the vertices in $D^{\delta}$ can be covered by any of
the $\delta$ FC trees, which finishes our proof.




\section{Spanners with high treewidth}
\label{sec:spanners-with-high}


In this section we show that for any constants $k \ge 2$ and $c \ge 1$
there are graphs of treewidth $k$ such that no spanning subgraph of
treewidth $k-1$ can be a $c$-spanner of such a graph.  This result is
known for $k=2$ \cite{KLM}; in particular, a family of treewidth 2
graphs is constructed and it is shown that for any $c$ there is such a
graph $G_c$ where no spanning tree
$c$-spans $G_c$.  The proof in \cite{KLM} is quite complicated and we
first present the sketch of an easy proof of the $ k = 2$ case; this
proof provides insight into the proof for arbitrary $ k > 2$.

First we define a particular graph $G$, of treewidth $k$, called a
$(k,h)$-snowflake, where $h$, a function of $k$ and $c$, denotes the
\emph{height} of the snowflake.  We will show that no spanning
subgraph (of $G$) having treewidth $k-1$ can be an additive
$c$-spanner for $G$.  For the proof we will assume the existence of an
arbitrary but fixed subgraph $T$ with treewidth at most $k-1$ that
$c$-spans $G$ and prove that there is always an edge $u v$ in $G
\setminus T$ where the shortest path between $u$ and $v$ in $T$ is
greater than $c$, thus implying that $T$ is not a $c$-spanner for $G$.
For an edge $uv$ in $G\setminus T$ we define its \emph{deficiency}
(denoted $\deficiency (uv)$) to be $\dist_{T}(u,v) - 1$.  Any edge in
$G\setminus T$ with deficiency greater than $c$ implies that $T$ is
not a $c$-spanner.


A snowflake is a special case of a $k$-tree, a well-known family of
graphs of treewidth $k$, given by the following recursive definition:

\begin{definition}\label{def:k-tree}
  For integer $k \ge 1$, a graph is a \emph{$k$-tree} if and only if
  it can be constructed in the following way:
  \begin{enumerate}[1)]
  \item~\label{item:1} A $(k+1)$-clique $K_B$ is a $k$-tree.
  \item If $G$ is a $k$-tree and $K$ is a $k$-clique
    of $G$, then adding a vertex $z$ to $G$ together with all edges
    between $z$ and the vertices of $K$ is again a $k$-tree.
  \end{enumerate}
\end{definition}
Note that this definition varies slightly from the standard definition
where the initial clique is a $k$-clique rather than a $(k+1)$-clique.
Our definition does not allow a $k$-clique itself to be considered a
$k$-tree, which is of no consequence for our arguments.
We call the initial $(k+1)$-clique, $K_B$, the \emph{base clique} of
$G$ and assume that $K_B$ is on vertices $\{v_1, v_2, \cdots,
v_{k+1}\}$ and that the rest of the vertices of $G$ are numbered
$v_{k+2}, v_{k+3}, \cdots, v_n$ in the order in which they are added
to $G$.

\begin{claim}
  For all $c \ge 1$, there is a graph $G_c$ of treewidth 2 such that
  no spanning tree of $G_c$ is a $c$-spanner of $G_c$.
\end{claim}

\begin{proof}[Proof (sketch)]
  Consider the $2$-tree $G_c$ constructed as follows, where $n = 3
  \cdot 2^{c-1}$: $v_4, v_5, v_6$ each form a triangle with one of the
  $3$ edges of $K_B$, thereby adding $6$ new edges; $v_7, \cdots,
  v_{12}$ each form a triangle with one of these $6$ new edges,
  thereby adding $12$ new edges; etc.

  Note that for $c = 1$, $G_c$ is a triangle and the claim is
  trivially true; we assume $c >1$.  By the construction of $G_c$,
  each edge $e$ of $K_B$ is a separator of $G_c$.  Note that in the
  construction of $G_c$ one of the vertices $v_4, v_5,$ or $v_6$ was
  added as the vertex that forms a triangle with $e$; without loss of
  generality assume that $v_4$ is this vertex and $e$ is the edge $v_2
  v_3$.  We let $S_1$ denote the induced subgraph of $G_c$ formed on
  the connected component of $G_c \setminus \{v_2, v_3\}$ that
  contains $v_4$ together with the edge $v_2 v_3$.  Note that $v_1$ is
  not in $S_1$.  In a similar way we define the subgraphs $S_2$ and
  $S_3$, where $S_2$ does not contain $v_2$ and $S_3$ does not contain
  $v_3$.

  Now consider any spanning tree $T$ of $G_c$.  To avoid a cycle in
  $T$ it cannot contain both a $v_2$ to $v_3$ path in
  $S_1$ and a $v_1$ to $v_3$ path in $S_2$ and a $v_2$ to $v_1$ path
  in $S_3$.  Without loss of generality, assume that there is no $v_2$
  to $v_3$ path in $S_1$.  Now the edge $v_2 v_3$ has deficiency at
  least 1.  Note that $S_1$ with edge $v_2 v_3$ removed consists of
  $S_4$ that contains edge $v_2 v_4$ and $S_5$ that contains edge $v_3
  v_4$.  To avoid $S_1$ having a $v_2$ to $v_3$ path either there is
  no $v_2 v_4$ path in $S_4$ or no $v_3 v_4$ path in $S_5$; without
  loss of generality assume the former.  Now the deficiency of the
  edge $v_2 v_4$ is at least one bigger than the deficiency of the
  edge $v_2 v_3$.  A simple induction argument completes the proof.
\end{proof}

For higher values of $k$ the identification of edges of $G \setminus
T$ that result in a higher deficiency is more complicated since we
have to follow a path of cliques to a clique that contains an edge of
increased deficiency; the length of this path increases with $k$ and
the current lower bound on the deficiency.

To aid the reader, we divide the rest of the section into the
following three subsections:
\begin{enumerate}
\item The $(k,h)$-snowflake and terminology

\item Showing the existence of an edge with deficiency at least 1

\item Showing the existence of an edge of arbitrarily high deficiency
\end{enumerate}

\subsection{The $(k,h)$-snowflake and terminology}\label{term}

First we examine the \emph{construction tree $\T$ of
  $k$-tree $G$}, a notion that follows naturally
from the recursive definition of $k$-trees.  The nodes of the
construction tree $\T$ consist of the root node $K_B$ and $n-k-1$
nodes corresponding to the $(k+1)$-cliques created by the addition of
$v_i, \: k+2 \le i \le n$; we let node $v_i$ in $\T$ represent the
$(k+1)$-clique created with the inclusion of the vertex $v_i$.
Furthermore, let $p(i)$ be the largest (according to the numbering of
the vertices) vertex in the $k$-clique to which $v_i$ is connected on
introduction to $G$.  Now in $\T$, $v_i$ is a child of $p(i)$; note
that if $p(i) = k$ or $k+1$, then $v_i$ is a child of $K_B$.
For $K$, a $(k+1)$-clique of $G$, $ K \neq K_B$ , we let $v_K$ denote
the vertex in $K$ of highest index and use $\T(K)$ to refer to the
subtree of the construction tree $\T$ that is rooted at node $v_K$;
$G[\T(K)]$ (respectively $T[\T(K)]$) refers to the
corresponding induced subgraph of $G$ (respectively
  $T$).  Furthermore, in $G$, $K$ is a separator between non-$K$
vertices introduced in $\T \setminus \T(K)$ and non-$K$ vertices
introduced in $\T(K)$.  Note that for each rooted tree $\T$, one can
easily construct a $k$-tree $G$ such that $\T$ is the construction
tree of $G$; however, $G$ is not in general uniquely determined.  In
$\T$, if $v_{K'} = v_i$ is a child of $v_K$, then there is a unique
vertex $x \in K \setminus K'$ and we refer to $v_i$ and $K'$ as a
\emph{$-x$~child of $K$}.  In the $k$-trees considered in our
arguments, $K_B$ will have $\lceil k/2 \rceil$ $-x$ children for each
$x \in K_B$ whereas all other $(k+1)$-cliques $K$ will have either 1
or 0 $-x$ children for each $x \in K$.






\begin{figure}
  \begin{center}
    \begin{minipage}{0.4\linewidth}
      \includegraphics[scale=0.35]{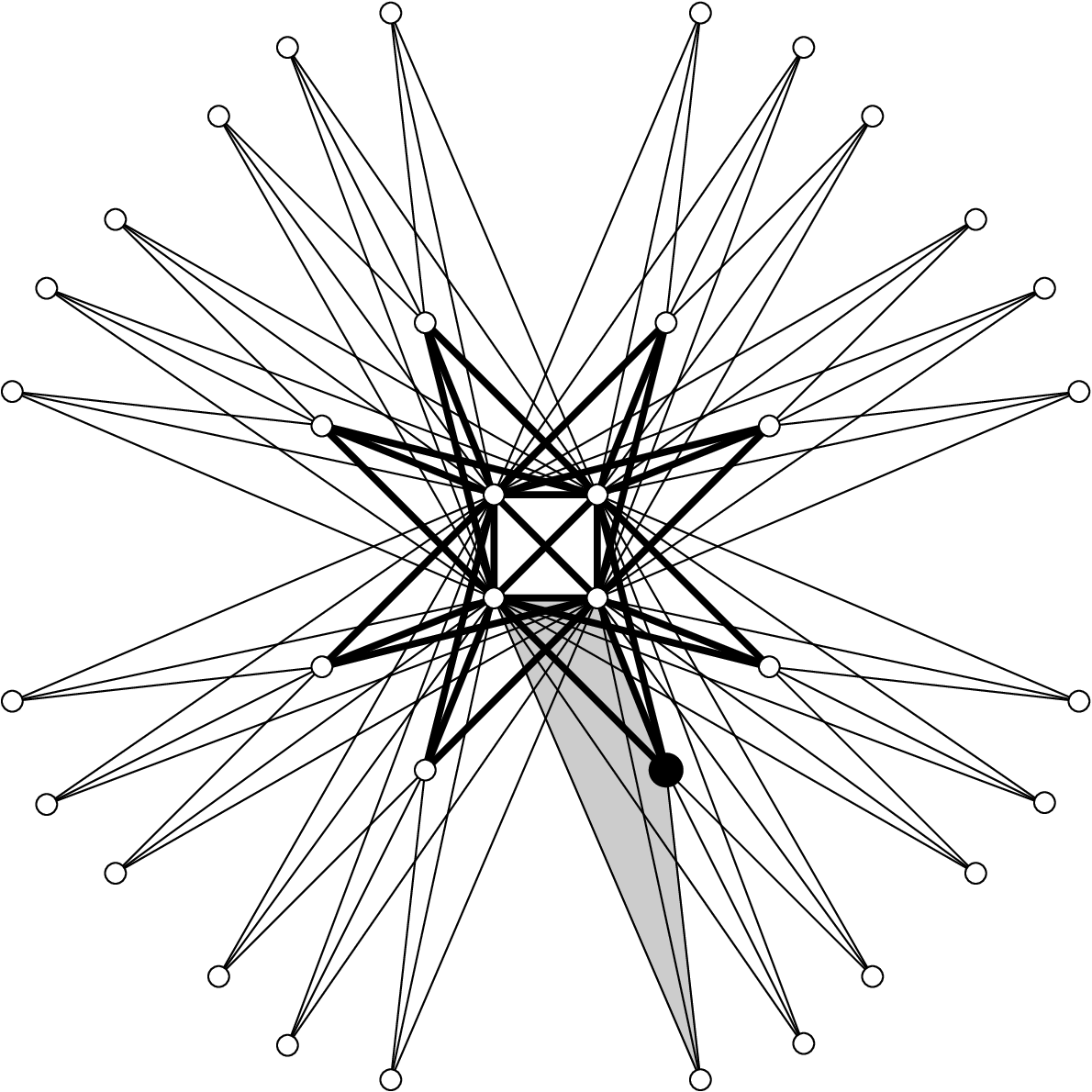}
    \end{minipage}
    \hspace{0.1\linewidth}
    \begin{minipage}{0.4\linewidth}
      \psfrag{KB}{$K_B$}%
      \includegraphics[scale=0.35]{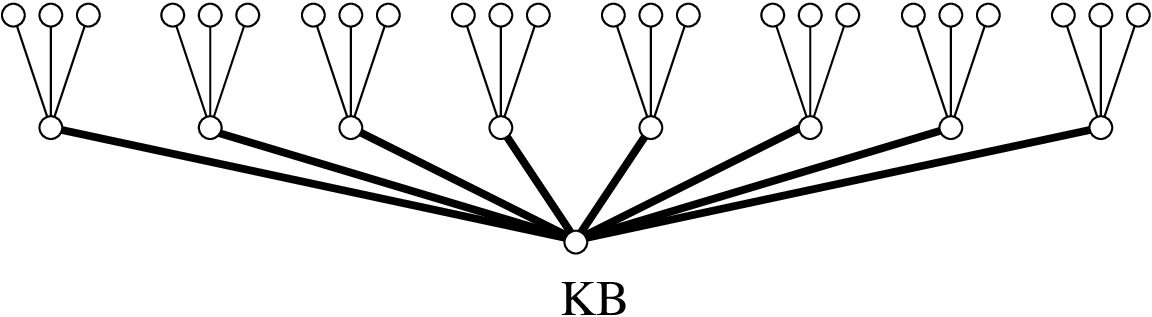}
    \end{minipage}
    \caption{A $(3,1)$-snowflake (bold edges) is shown that is
      contained in a $(3,2)$-snowflake together with the construction
      tree of the $(3,1)$-snowflake (bold edges) contained in the
      construction tree of the $(3,2)$-snowflake.  The corner vertices
      of the shaded area form a clique $K$ with the black vertex being
      the earliest vertex of
      $K$.}
    \label{fig:4-tree-together}
  \end{center}
\end{figure}

Our argument will be based on a specific $k$-tree, namely a
\emph{$(k,h)$-snowflake}, $k \ge 3, h \ge 0$, denoted $G = S(k,h)$ or
simply $G$, if there is no chance for ambiguity.  Note that
the parameter $h$ is the \emph{height} of the
snowflake.  The base $(k+1)$-clique $K_B$ is the $S(k,0)$-snowflake
and $K_B$ is at height 0. $S(k,1)$ consists of $S(k,0)$ as well as
$\lceil k/2 \rceil$ $-v$ children (all of which will be called a $-v$
child of $K_B$) for each $v \in K_B$.  This is done to ensure that
$K_B$ has at least ${k+1 \choose 2}$ children, i.e.,
at least one child for each pair of distinct vertices of $K_B$.  For
$1<i\le h$, the $S(k,i)$-snowflake is constructed from the
$S(k,i-1)$-snowflake in the following way.  For every $(k+1)$-clique
$K$ (where $v_j$ is the vertex whose introduction created $K$) at
height $i-1$ and every vertex $v \neq v_j$ in $K$ we construct exactly
one $-v$ child of $K$ and each of these children is at height $i$.
Thus in the construction trees of this restricted family of $k$-trees,
the root $K_B$ has $(k+1) \cdot \lceil k/2 \rceil$ children, each
vertex at height between $1$ and $h-1$ inclusive has exactly $k$
children and each vertex at height $h$ has no children.  See
Figure~\ref{fig:4-tree-together} for an example of a $(3,1)$- and a
$(3,2)$-snowflake together with its construction tree.


Note for an arbitrary $(k,h)$-snowflake,
given a $(k+1)$-clique $K$ at height at least $k-1$ we refer to the
\emph{earliest vertex} of $K$ as the non-$K_B$ vertex of $K$ that was
introduced at the lowest height in the snowflake.

\subsection{Showing the existence of an edge with deficiency at least
  1}\label{base-case}

To show the existence of an edge with deficiency at least 1, we have
to describe a chain of cliques in $\T$ starting at $K_B$ and ending at
an ``$x,y$-successor of $K_B$'' where $x,y$ is an arbitrary pair of
vertices in $K_B$.  Subsequent examination of the set of all
$u,v$-successors of $K_B$ (where $u, v \in K_B$) will show that for
some $x, y$ in $K_B$ the edge $xy$ is in $G \setminus T$ and has
deficiency at least 1.

For the base $(k+1)$-clique $K_B$ of $(k,h)$-snowflake $G$ (of
sufficient height $h$) with vertices $\{a_1, a_2, \dots, a_{k-1}, x, y
\}$, an \emph{$x,y$-successor of $K_B$} is any $(k+1)$-clique in
$\T(K_B)$ determined as follows.  Let $\pi$ be a permutation of $\{1,
2, \dots, k-1\}$.  $K^{\pi}_0$ is defined to be $K_B$ and $K^{\pi}_1$
is defined to be the $k+1$-clique created from $K_B$ by replacing
vertex $a_{\pi(1)}$ with one of the $\lceil k/2 \rceil$ $-a_{\pi(1)}$
children.  For all $1 < i \le k-1$, $K^{\pi}_i$ is defined to be the
$(k+1)$-clique created from $K^{\pi}_{i-1}$ by replacing vertex
$a_{\pi(i)}$ with new vertex $-a_{\pi(i)}$.  $K^{\pi}_{k-1}$ is an
\emph{$x,y$-successor of $K_B$} and is denoted by $K^{\pi}$.  Now
suppose we take another pair of vertices $u,v$ in $K_B$ and choose
another permutation $\pi'$ for this pair.  If the first elements of
$\pi$ and $\pi'$ are different, or the first elements are the same but
different $-a_{\pi(1)}$ children are chosen, then $K^{\pi} \cap
K^{\pi'} = \{x,y\} \cap \{u,v\}$.  Furthermore, for such a pair of
permutations, any $x,y$-path in $G[\T(K^{\pi})]$
is internally disjoint from any $u,v$-path in $G[\T(K^{\pi'})]$.
For any such permutation $\pi$, for a pair of vertices $x,y$ in $K_B$,
the $x,y$-successor $K^{\pi}$ of $K_B$ defines a subgraph
$G[\T(K^{\pi})]$ that we call an \emph{$x y$-branch of $K_B$}.

In the following, we assume that we are dealing with a
$(k,h)$-snowflake for sufficiently large $h$.  Now suppose there is a
spanner $T$ of $G$ having treewidth less
than or equal to $k-1$.

\begin{fact}\label{fact:1}
  Consider the $(k+1)$-clique $K_B$ of a $(k,h)$-snowflake for
  sufficiently large $h$.  There exists a pair of vertices $x, y$ of
  $K_B$ and an $x y$-branch of $K_B$ that does not contain an
  $x,y$-path in $T$.  Thus the edge $xy$ is in $G \setminus T$ and has
  deficiency at least 1.
\end{fact}

\begin{proof}
  Suppose $K_B$ consists of vertices $\{a_1, a_2, \cdots a_{k+1}\}$.
  Consider all ${k+1 \choose 2}$ pairs of vertices, $a_i, a_j,\: 1 \le
  i < j \le k+1$.
  For each such pair, we identify a permutation $\pi_{i,j}$ of the
  vertices of $K_B \setminus \{a_i,a_j\}$ where the first entry of the
  permutation is $a_\ell, \:\ell \notin \{i,j\}$ and no more than
  $\lceil k / 2 \rceil$ of the chosen permutations start at $a_\ell$.
  Now each of these at most $\lceil k / 2 \rceil$ permutations whose
  child of $K_B$ will be a $-a_\ell$ vertex, chooses a unique
  $-a_\ell$ child from the $\lceil k / 2 \rceil$ children available.
  Doing so, we continue and arrive at an $a_i, a_j$-successor of
  $K_{B}$.  Now suppose that each of the
  corresponding ${k+1 \choose 2}$ branches $\T(K^{\pi_{i,j}}), 1 \le i
  < j \le k+1$ contains an $a_i, a_j$-path in $T$.  By the choice of
  branches, these paths are pairwise internally disjoint.
  Thus $T$ contains a $K_{k+1}$ minor, contradicting $T$ having
  treewidth $k-1$.
  

  
  Therefore there exists a branch $\T(K^{\pi_{i,j}})$ (where $a_i = x$
  and $a_j = y$) such that $T$ restricted to this branch has no
  $x,y$-path, and in particular edge $x y$ is
  not in $T$ and thus has deficiency at least 1.
  
  \end{proof}

\subsection{Showing the existence of an edge of arbitrarily high
  deficiency}


%

We now show how to construct a chain of cliques from $K^{\pi}$, an
$x,y$-successor of $K_B$ to another clique where there is an edge with
its deficiency at least one more than the deficiency of $xy$ in
$K^{\pi}$.

A $(k+1)$-clique $K$ containing vertices
$x,y$ is said to satisfy the
\emph{$x,y$-path property}
if $T[\T(K)]$ does not contain an $x,y$-path.  In the following, as
guaranteed by Fact~\ref{fact:1}, we will assume that $K^{\pi}$ is an
$x,y$-successor of $K_B$ that satisfies the $x,y$-path property.
Based on the existence of $K^{\pi}$, we will identify a sufficiently
long chain of $(k+1)$-cliques starting at $K^{\pi}$.  The following
definition shows how to ``build'' such a chain off an arbitrary
$(k+1)$-clique $K$ containing $x,y$ that satisfies
the $x,y$-path property.


\begin{definition}\label{chain}
  Let $K$ be a $(k+1)$-clique of height at least $k-1$, such that $K$
  satisfies the $x,y$-path property for $x,y \in K$.  An \emph{$x,y$
    chain of cliques starting from $K$, for arbitrary $i \ge 0$} is
  defined as follows:
  \begin{eqnarray*}
    K^0 &=& K\\
    K^i &=& -u\text{ child of $K^{i-1}$ where $u$ is the earliest introduced vertex}\\
    & &\text{ (with respect to $\T$) in $K^{i-1} \setminus\{x, y\}$, for $i \ge 1$.}
  \end{eqnarray*}
  The new vertex created for $K^i$ ($i \ge 1$) is denoted as $z_i$.
  See Figure~\ref{fig:fact2} for an $x,y$ chain of cliques built off
  $K^{\pi}$, an $x, y$-successor of $K_B$.
\end{definition}

\begin{fact}\label{fact:2}
  Let $z_i$ be defined as above.  Then $\dist_{G\setminus\{x,y\}}(z_i,
  K\setminus\{x,y\}) = 1 + \left\lfloor \frac{i-1}{k-2}
  \right\rfloor$.
\end{fact}

\begin{proof}
  By the tree-like structure of $G$,
  any internal vertices of a shortest path from $z_i$ to
  $K\setminus\{x, y\}$ in $G\setminus\{x, y\}$
  has to consist entirely of $z_j$ vertices of this chain of cliques
  leading from $K$ to $z_i$
  with $0 < j < i$.  Any consecutive set of $k-1$ such $z$ vertices
  forms a clique and there are no edges between vertices $z_{j_1}$ and
  $z_{j_2}$ with $|j_2 - j_1| > k-2$.

  By construction, vertices $z_1$ to $z_{k-2}$ have neighbors in
  $K\setminus \{x,y\}$ and thus have distance $1$ to $K\setminus\{x,
  y\}$.
  Similarly, vertices $z_{k-1}$ to $z_{2(k-2)}$ have distance $2$ to
  $K\setminus\{x, y\}$.  Extending this argument inductively shows
  that vertices $z_{(\ell -1)(k-2)+1}$ to $z_{\ell(k-2)}$ have
  distance $\ell$ to $K\setminus\{x, y\}$, implying that vertex $z_i$
  has distance $1 + \left\lfloor \frac{i-1}{k-2} \right\rfloor$ to
  $K\setminus\{x, y\}$ in $G\setminus\{x,y\}$.
\end{proof}



\begin{figure}
 \begin{center}
   \psfrag{ppp}{\scriptsize$\dots$}%
   \psfrag{K}{\scriptsize$K_{B}\setminus\{x, y\}$}%
   \psfrag{Kpi}{\scriptsize$K^{\pi}\setminus\{x, y\}$}%
   \psfrag{K0}{\scriptsize$=K^0\setminus \{x, y\}$}%
   \psfrag{Ka}{\scriptsize$K^{\alpha}\setminus\{x, y\}$}%
   \psfrag{x}{\scriptsize$x$}%
   \psfrag{y}{\scriptsize$y$}%
   \psfrag{a1}{\scriptsize$a_1$}%
   \psfrag{a2}{\scriptsize$a_2$}%
   \psfrag{ak-1}{\scriptsize$a_{k-1}$}%
   \psfrag{-a1}{\scriptsize$-a_1$}%
   \psfrag{-a2}{\scriptsize$-a_2$}%
   \psfrag{-ak-1}{\scriptsize$-a_{k-1}$}%
   \psfrag{z0}{\scriptsize$-a_{k-1}$}%
   \psfrag{z1}{\scriptsize$z_1$}%
   \psfrag{z2}{\scriptsize$z_2$}%
   \psfrag{zi}{\scriptsize$z_i$}%
   \psfrag{za}{\scriptsize$z_{\alpha}$}%
   \psfrag{zak2}{\scriptsize$z_{\alpha-k+2}$}%
   \includegraphics[width=14cm]{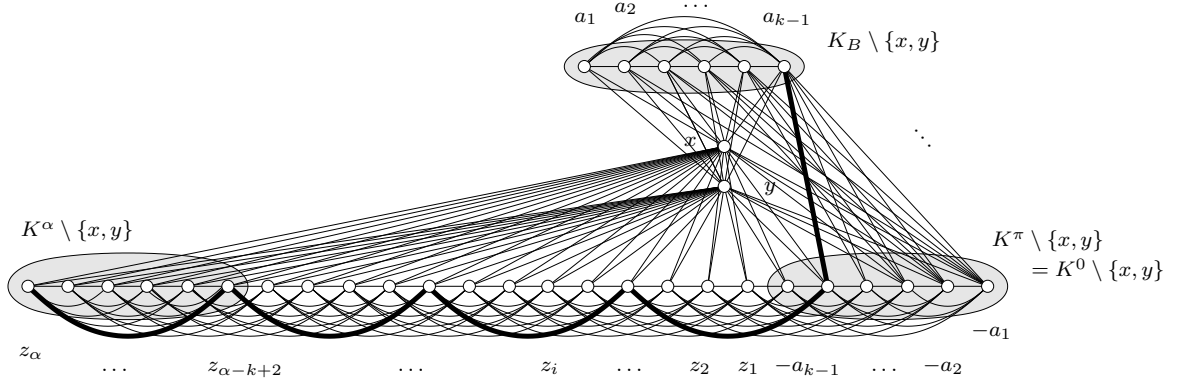}
 \end{center}
 \caption{The diagram shows a chain of cliques in a $7$-snowflake of
   sufficient height built off $K^{\pi}$, an $x,y$-successor of $K_B$.
   Further the diagram shows vertices $z_1, \dots, z_i, \dots
   z_{\alpha}$ as defined in Definition~\ref{chain} and as used in
   Fact~\ref{fact:2}, as well as a clique $K^{\alpha}$ used in
   Lemma~\ref{lem:1}.  For $z_{\alpha}$ the shortest path to
   $K_B\setminus\{x,y\}$ in
   $G\setminus\{x,y\}$ is depicted by bold edges.}
 \label{fig:fact2}
\end{figure}



\begin{lemma}\label{lem:1}
  Let $K$, at height at least $k-1$, be a $(k+1)$-clique containing
  $x,y$ that satisfies the $x,y$-path property, where
  $\deficiency(x,y) \geq C \geq 1$.
  Let $K^{\alpha}$ be the $(k+1)$-clique in the chain of cliques from
  $K$ (see Definition~\ref{chain}) where $\alpha = 1+(C-1)(k-2)$.
%
  Then there exists $w \in \{x, y\}$ such that $K^{\alpha}$
  (containing $z_{\alpha},w$) satisfies the $z_{\alpha}, w$-path
  property and $\deficiency(z_{\alpha}, w) \geq C+1$.
\end{lemma}

\begin{proof}
  By the $x,y$-path property on $K$, $K$ must satisfy at least one of
  $z_{\alpha}, x$- and $z_{\alpha}, y$-path properties.  Without loss
  of generality suppose $K$ satisfies the $z_{\alpha}, x$-path
  property which implies $z_{\alpha} x \notin E_T$.  Thus, since
  $K^{\alpha}$ is in $\T(K)$ and contains $z_{\alpha}$ and $x$,
  $K^{\alpha}$ also satisfies the $z_{\alpha}, x$-path property.


  Now we just need to show that $\deficiency(z_{\alpha}, x) \geq C+1$.
  Note that there may or may not exist a $z_{\alpha},y$-path in
  $T[\T(K)]$; in particular it is possible that $z_{\alpha} y \in
  E_T$.

  We now look at the two ways of forming a shortest
  $z_{\alpha},x$-path in $T$.

  \textbf{Case 1:} \emph{The shortest $z_{\alpha},x$-path in $T$
    contains vertex $y$.}\\
  The subpath between $z_{\alpha}$ and $y$ has length $\geq 1$.  Since
  $\deficiency(x,y) \geq C$ the distance between $x$ and $y$ in $T$ is
  greater than $C$ and thus $\deficiency(z_{\alpha}, x) \geq C+1$ as
  required.

  \textbf{Case 2:} \emph{The shortest $z_{\alpha},x$-path in $T$ does not contain $y$.}\\
  Since $K$ separates $G[\T(K)]$ from the rest of the graph, any
  $z_{\alpha},x$-path in $T$ has to pass through $K\setminus\{x,y\}$.
  Let $a_j$ be the first vertex of this path in the $z_{\alpha}$ to
  $x$ direction where $a_j \in K\setminus \{x,y\}$.
  Now, $$\dist_T(z_{\alpha}, a_j) \geq \dist_{G\setminus\{x,
    y\}}(z_{\alpha}, a_j) \geq C$$ (by Fact~\ref{fact:2}) which
  implies

  $$\dist_T(z_{\alpha},x) > C+1$$ as required.
\end{proof}


\begin{theorem}\label{thm:main-theorem}
  For any constants $c \ge 1$ and $k \ge 3$, there exists a $k$-tree
  $G$ such that there is no $+c$ spanner $T$
  that has treewidth less than or equal to $k-1$.
\end{theorem}

\begin{proof}
  Let $G$ be the $(k, k c^2)$-snowflake
  $S(k, k c^{2})$.  Fact~\ref{fact:1} guarantees the existence of
  $K^{\pi}$ (of height $k-1$) where there is a pair of vertices $x, y
  \in K^{\pi}$ with $K^{\pi}$ satisfying the
  $x,y$-path property;
  clearly the deficiency $\deficiency(x,y) \geq 1$.

  Applying Lemma~\ref{lem:1} to $K^{\pi}$ and this particular pair of
  vertices $x, y$, we identify a $(k+1)$-clique $K^{\alpha}$ and
  vertices $z_{\alpha}, w$, where $w \in \{x, y\}$ such that
  $K^{\alpha}$ satisfies the $z_{\alpha}, w$-path property and
  $\deficiency(z_{\alpha}, w) \geq 2$.  Continuing to iteratively
  apply Lemma~\ref{lem:1} $c-1$ times to the $(k+1)$-cliques and
  identified pairs of vertices, we prove the existence of an edge of
  deficiency greater than $c$.

  We now show that $kc^2$ is a sufficient height for the snowflake.
  By Lemma~\ref{lem:1}, each of the at most $c-1$ iterations requires
  the height of the snowflake to be increased by $1+(C-1)(k-2)$ for
  $C$ ranging from 1 to $c-1$.  Thus height $kc^2$ suffices.
\end{proof}



{\small%
\begin{figure}
 \begin{center}
   \psfrag{ppp}{\scriptsize$\dots$}%
   \psfrag{K}{\scriptsize$K_{B}\setminus\{x, y\}$}%
   \psfrag{Kpi}{\scriptsize$K^{\pi}\setminus\{x, y\}$}%
   \psfrag{K0}{\scriptsize$=K^0\setminus \{x, y\}$}%
   \psfrag{Ka}{\scriptsize$K^{\alpha}\setminus\{x, y\}$}%
   \psfrag{x}{\scriptsize$x$}%
   \psfrag{y}{\scriptsize$y$}%
   \psfrag{a1}{\scriptsize$a_1$}%
   \psfrag{a2}{\scriptsize$a_2$}%
   \psfrag{ak-1}{\scriptsize$a_{k-1}$}%
   \psfrag{-a1}{\scriptsize$-a_1$}%
   \psfrag{-a2}{\scriptsize$-a_2$}%
   \psfrag{-ak-1}{\scriptsize$-a_{k-1}$}%
   \psfrag{z0}{\scriptsize$-a_{k-1}$}%
   \psfrag{z1}{\scriptsize$z_1$}%
   \psfrag{z2}{\scriptsize$z_2$}%
   \psfrag{zi}{\scriptsize$z_i$}%
   \psfrag{za}{\scriptsize$z_{\alpha}$}%
   \psfrag{zak2}{\scriptsize$z_{\alpha-k+2}$}%
   \psfrag{z1s}{\scriptsize$z'_1$}%
   \psfrag{z2s}{\scriptsize$z'_2$}%
   \psfrag{zas}{\scriptsize$z'_{\alpha}$}%
   \psfrag{zak2s}{\scriptsize$z'_{\alpha-k+2}$}%
   \includegraphics[width=14cm]{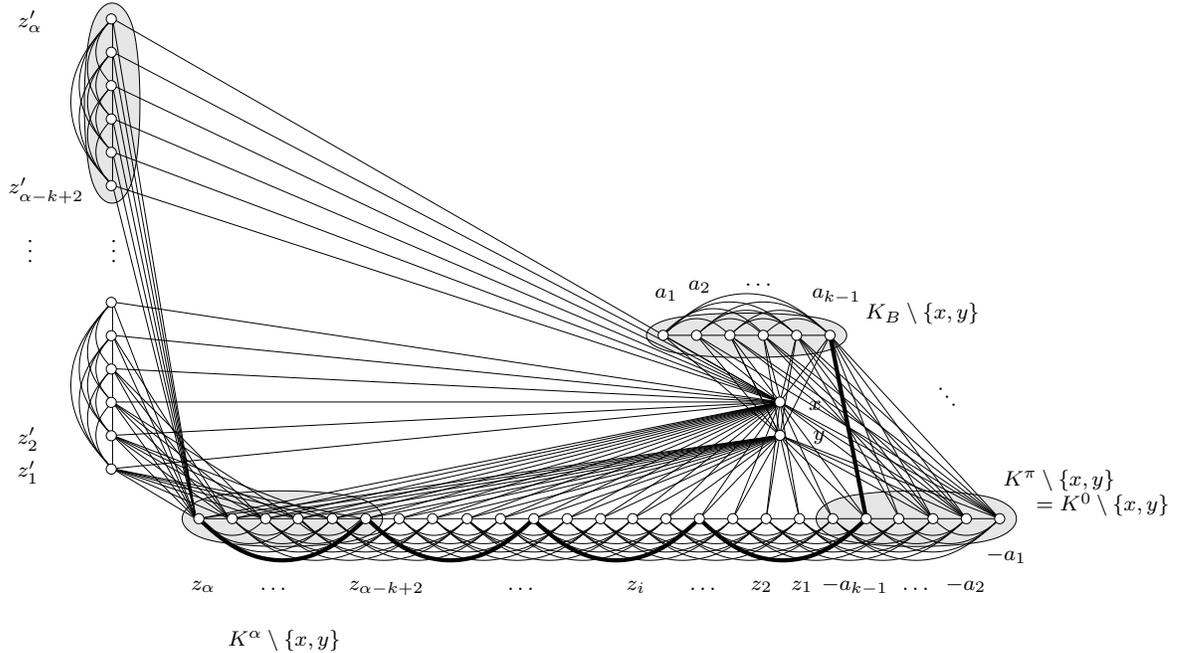}
 \end{center}
 \caption{The diagram shows the next induction step following the
   proof of Theorem~\ref{thm:main-theorem} of the graph in
   Figure~\ref{fig:fact2}.  Vertex $x$ takes over the role of $w$.}
 \label{fig:main-theorem-ind}
\end{figure}
}%

\section{Concluding Remarks}\label{concl}


There are a number of immediate questions building off our results:
\begin{itemize}
\item Can our lower bound techniques be used to prove other lower
  bound results on collective tree spanners and on spanners of bounded
  treewidth?

\item For chordal graphs are there any values of $c > 3$ where there
  is no constant sized set of spanning trees that collectively $+c$
  spans?

\item For any families presented in Table \ref{tbl:summary} can the
  gap between the lower and upper bounds for the size of systems of
  collective additive tree spanners be tightened?

\item Are there families of graphs that admit (do not admit)
  $O(1)$-spanners with treewidth 2 (or 3, or any small constant)?

\end{itemize}

More generally, looking at the results presented in Table
\ref{tbl:summary} we can state that a family of graphs belongs to
${\cal F}(s(n),\mu(n))$ if every graph $G$ from ${\cal
  F}(s(n),\mu(n))$ admits a system of at most $\mu(n)$ collective
additive tree $s(n)$-spanners.  From this perspective all interval
graphs, AT-free graphs, strongly chordal graphs belong to ${\cal
  F}(O(1),O(1))$ meaning that every graph $G$ from these families can
be $+c$-spanned, for some constant $c$, with at most a constant number
of spanning trees. All graph classes from Table \ref{tbl:summary}
belong to ${\cal F}(O(1),O(\log n))$, however weakly chordal graphs
and outerplanar graphs are in fact in ${\cal F}(O(1),O(\log
n))\setminus {\cal F}(O(1),O(1))$ while for chordal graphs this is
unknown. The open question mentioned earlier for chordal graphs can be
stated as follows: Is the class of chordal graphs in ${\cal
  F}(O(1),O(1))$?  Also, what other known graph families belong to
${\cal F}(O(1),O(1))$ or ${\cal F}(O(1),O(\log n))\setminus {\cal
  F}(O(1),O(1))$? More generally, given a family of graphs,
how many spanning trees are necessary/sufficient to collectively
$O(1)$-span any graph in the family?
   

%
%
%
%

\bigskip{}

\paragraph{Acknowledgements:}
DGC wishes to thank the Natural Sciences and Engineering Research
Council of Canada for financial assistance in the support of this
research.


\bibliography{05-lower-bound}

\begin{thebibliography}{10}

\bibitem{bartal1996probabilistic}
{\sc Y.~Bartal}, {\em Probabilistic approximation of metric spaces and its algorithmic applications}, in Proceedings of 37th Conference on Foundations of Computer Science, IEEE, 1996, pp.~184--193.

\bibitem{BARTAL202226}
{\sc Y.~Bartal, O.~N. Fandina, and O.~Neiman}, {\em Covering metric spaces by few trees}, Journal of Computer and System Sciences, 130 (2022), pp.~26--42.

\bibitem{BeRa22}
{\sc O.~Bendele and D.~Rautenbach}, {\em Additive tree ${O}(\rho \log n)$-spanners from tree breadth $\rho$}, Theoretical {C}omputer {S}cience,  (2022), pp.~39--46.

\bibitem{BrandstadtCD99}
{\sc A.~Brandst{\"{a}}dt, V.~Chepoi, and F.~F. Dragan}, {\em Distance approximating trees for chordal and dually chordal graphs}, J. Algorithms, 30 (1999), pp.~166--184.

\bibitem{Classes}
{\sc A.~Brandst{\"{a}}dt, V.~B. Le, and J.~P. Spinrad}, {\em Graph Classes : a Survey}, SIAM Monographs on Discrete Mathematics and Applications, SIAM, 1999.

\bibitem{CaiC95}
{\sc L.~Cai and D.~G. Corneil}, {\em Tree spanners}, SIAM J. Disc. Math., 8 (1995), pp.~359--387.

\bibitem{10353154}
{\sc H.-C. Chang, J.~Conroy, H.~Le, L.~Milenkovic, S.~Solomon, and C.~Than}, {\em { Covering Planar Metrics (and Beyond): O(1) Trees Suffice }}, in 2023 IEEE 64th Annual Symposium on Foundations of Computer Science (FOCS), Los Alamitos, CA, USA, Nov. 2023, IEEE Computer Society, pp.~2231--2261.

\bibitem{chang_et_al:LIPIcs.SoCG.2024.37}
{\sc H.-C. Chang, J.~Conroy, H.~Le, L.~Milenkovi\'{c}, S.~Solomon, and C.~Than}, {\em {Optimal Euclidean Tree Covers}}, in 40th International Symposium on Computational Geometry (SoCG 2024), W.~Mulzer and J.~M. Phillips, eds., vol.~293 of Leibniz International Proceedings in Informatics (LIPIcs), Dagstuhl, Germany, 2024, Schloss Dagstuhl -- Leibniz-Zentrum f{\"u}r Informatik, pp.~37:1--37:15.

\bibitem{chang2025lighttreecoversrouting}
{\sc H.-C. Chang, J.~Conroy, H.~Le, S.~Solomon, and C.~Than}, {\em Light tree covers, routing, and path-reporting oracles via spanning tree covers in doubling graphs}, 2025.

\bibitem{charikar1998approximating}
{\sc M.~Charikar, C.~Chekuri, A.~Goel, S.~Guha, and S.~Plotkin}, {\em Approximating a finite metric by a small number of tree metrics}, in Proceedings 39th Annual Symposium on Foundations of Computer Science (Cat. No. 98CB36280), IEEE, 1998, pp.~379--388.

\bibitem{Chew}
{\sc L.~Chew}, {\em There are planar graphs almost as good as the complete graph}, J. of Computer and System Sciences, 39 (1989), pp.~205--219.

\bibitem{CohenAddad2020OnLS}
{\sc V.~Cohen-Addad, A.~Filtser, P.~N. Klein, and H.~Le}, {\em On light spanners, low-treewidth embeddings and efficient traversing in minor-free graphs}, 2020 IEEE 61st Annual Symposium on Foundations of Computer Science (FOCS),  (2020), pp.~589--600.

\bibitem{10353171}
{\sc V.~Cohen-Addad, H.~Le, M.~Pilipczuk, and M.~Pilipczuk}, {\em { Planar and Minor-Free Metrics Embed into Metrics of Polylogarithmic Treewidth with Expected Multiplicative Distortion Arbitrarily Close to 1* }}, in 2023 IEEE 64th Annual Symposium on Foundations of Computer Science (FOCS), Los Alamitos, CA, USA, Nov. 2023, IEEE Computer Society, pp.~2262--2277.

\bibitem{CollinsCGKM11}
{\sc A.~Collins, J.~Czyzowicz, L.~Gasieniec, A.~Kosowski, and R.~A. Martin}, {\em Synchronous rendezvous for location-aware agents}, in Distributed Computing - 25th International Symposium, {DISC} 2011, Rome, Italy, September 20-22, 2011. Proceedings, 2011, pp.~447--459.

\bibitem{CorneilDKY05}
{\sc D.~G. Corneil, F.~F. Dragan, E.~K{\"{o}}hler, and C.~Yan}, {\em Collective tree 1-spanners for interval graphs}, in Graph-Theoretic Concepts in Computer Science, 31st International Workshop, {WG} 2005, Metz, France, June 23-25, 2005, Revised Selected Papers, 2005, pp.~151--162.

\bibitem{COS96}
{\sc D.~G. Corneil, S.~Olariu, and L.~Stewart}, {\em Asteroidal {T}riple--free {G}raphs}, SIAM J. Disc. Math.,  (1997), pp.~399--430.

\bibitem{DraganA14}
{\sc F.~F. Dragan and M.~Abu{-}Ata}, {\em Collective additive tree spanners of bounded tree-breadth graphs with generalizations and consequences}, Theor. Comput. Sci., 547 (2014), pp.~1--17.

\bibitem{DraganCKX12}
{\sc F.~F. Dragan, D.~G. Corneil, E.~K{\"{o}}hler, and Y.~Xiang}, {\em Collective additive tree spanners for circle graphs and polygonal graphs}, Discrete Appl. Math., 160 (2012), pp.~1717--1729.

\bibitem{DrFoGo}
{\sc F.~F. Dragan, F.~V. Fomin, and P.~A. Golovach}, {\em Spanners in sparse graphs}, J. of Computer and System Sciences, 77 (2011), pp.~1108--1119.

\bibitem{DraganY10}
{\sc F.~F. Dragan and C.~Yan}, {\em Collective tree spanners in graphs with bounded parameters}, Algorithmica, 57 (2010), pp.~22--43.

\bibitem{DYC}
{\sc F.~F. Dragan, C.~Yan, and D.~G. Corneil}, {\em Collective tree spanners and routing in {AT}-free related graphs}, J. Graph Algorithms Appl., 10 (2006), pp.~97--122.

\bibitem{DYL}
{\sc F.~F. Dragan, C.~Yan, and I.~Lomonosov}, {\em Collective tree spanners of graphs}, SIAM J. Disc. Math., 20 (2006), pp.~241--260.

\bibitem{DraganYX08}
{\sc F.~F. Dragan, C.~Yan, and Y.~Xiang}, {\em Collective additive tree spanners of homogeneously orderable graphs}, in {LATIN} 2008: Theoretical Informatics, 8th Latin American Symposium, B{\'{u}}zios, Brazil, April 7-11, 2008, Proceedings, 2008, pp.~555--567.

\bibitem{faber-strongly-chordal-83}
{\sc M.~Farber}, {\em Characterization of strongly chordal graphs}, Discrete Appl. Math., 43 (1983), pp.~173--189.

\bibitem{FL2022}
{\sc A.~Filtser and H.~Le}, {\em { Low Treewidth Embeddings of Planar and Minor-Free Metrics }}, in 2022 IEEE 63rd Annual Symposium on Foundations of Computer Science (FOCS), Los Alamitos, CA, USA, Nov. 2022, IEEE Computer Society, pp.~1081--1092.

\bibitem{FominGL11}
{\sc F.~V. Fomin, P.~A. Golovach, and E.~J. van Leeuwen}, {\em Spanners of bounded degree graphs}, Inform. Process. Lett., 111 (2011), pp.~142--144.

\bibitem{FrGa01}
{\sc P.~Fraigniaud and C.~Gavoille}, {\em Routing in trees}, in Proceedings of the 28th Int. Colloquium on Automata, Languages and Programming (ICALP 2001), Lecture Notes in Computer Science 2076, 2001, pp.~757--772.

\bibitem{Gol-book}
{\sc M.~C. Golumbic}, {\em Algorithmic Graph Theory and Perfect Graphs Ed. 2}, Annals of Discrete Mathematics 57, North Holland, 2004.

\bibitem{gupta2006oblivious}
{\sc A.~Gupta, M.~T. Hajiaghayi, and H.~R{\"a}cke}, {\em Oblivious network design}, in Proceedings of the seventeenth annual ACM-SIAM symposium on Discrete algorithm, 2006, pp.~970--979.

\bibitem{Gup}
{\sc A.~Gupta, A.~Kumar, and R.~Rastogi}, {\em Traveling with a pez dispenser (or, routing issues in {MPLS})}, SIAM J. Comput., 34 (2005), pp.~453--474.

\bibitem{GuptaKL13}
{\sc S.~Gupta, S.~Kamali, and A.~L{\'{o}}pez{-}Ortiz}, {\em On advice complexity of the k-server problem under sparse metrics}, in Structural Information and Communication Complexity - 20th International Colloquium, {SIROCCO} 2013, Ischia, Italy, July 1-3, 2013, Revised Selected Papers, 2013, pp.~55--67.

\bibitem{KLM}
{\sc D.~Kratsch, H.-O. Le, H.~M\"{u}ller, E.~Prisner, and D.~Wagner}, {\em Additive tree spanners}, SIAM J. Disc. Math., 17 (2003), pp.~332--340.

\bibitem{LB}
{\sc C.~Lekkerkerker and J.~Boland}, {\em Representation of a finite graph by a set of intervals on the real line}, Fund. Math., 51 (1962), pp.~45--64.

\bibitem{LieShe}
{\sc A.~Liestman and T.~Shermer}, {\em Additive graph spanners}, Networks, 23 (1993), pp.~343--364.

\bibitem{McK}
{\sc T.~A. McKee}, {\em personal communication to {E}. {P}risner}, 1995.

\bibitem{PeSc}
{\sc D.~Peleg and A.~A. Sch\"affer}, {\em Graph spanners}, J. Graph Theory, 13 (1989), pp.~99--116.

\bibitem{PelUll}
{\sc D.~Peleg and J.~D. Ullman}, {\em An optimal synchronizer for the hypercube}, in Proc. 6th ACM Symposium on Principles of Distributed Computing, Vancouver, 1987, pp.~77--85.

\bibitem{stargraph}
{\sc E.~Prisner}, {\em Representing triangulated graphs in stars}, Abh. Math. Sem., 62 (1992), pp.~29--41.

\bibitem{ThZw01}
{\sc M.~Thorup and U.~Zwick}, {\em Compact routing schemes}, in Proceedings of the 13th Ann. ACM Symp. on Par. Alg. and Arch. (SPAA 2001), ACM, 2001, pp.~1--10.

\bibitem{YanXD12}
{\sc C.~Yan, Y.~Xiang, and F.~F. Dragan}, {\em Compact and low delay routing labeling scheme for unit disk graphs}, Comput. Geom., 45 (2012), pp.~305--325.

\end{thebibliography}
\end{document}